\newif\ifarxiv
\arxivtrue
\documentclass{article}

\ifarxiv
\usepackage[accepted]{icml2018_arxiv}
\else
\usepackage{icml2018}
\fi

\usepackage{natbib}
\usepackage{color, xcolor}
\usepackage{times,graphicx,amssymb,amsmath,natbib}
\definecolor{mydarkblue}{rgb}{0,0.08,0.45}
\definecolor{mydarkblue}{rgb}{0,0.08,0.45}
\usepackage[colorlinks=true,
    linkcolor=mydarkblue,
    citecolor=mydarkblue,
    filecolor=mydarkblue,
    urlcolor=mydarkblue]{hyperref}

\usepackage{dsfont}

\usepackage{amsthm}
\usepackage{mathtools}
\definecolor{ddarkbrown}{rgb}{0.5,0.2,0.05} \definecolor{bbluegray}{rgb}{0.05,0,0.5}

\definecolor{myblue}{HTML}{D2E4FC}
\newcommand*\mybluebox[1]{%
\colorbox{myblue}{\hspace{1em}#1\hspace{1em}}}

\newtheorem{theorem}{Theorem}[section]

\newtheorem{lemma}{Lemma}

\newenvironment{sketchproof}{\textbf{Proof sketch.}}{\QED\bigskip}
\renewenvironment{proof}{\textbf{Proof.}}{\QED\bigskip}

\newcommand{\BEAS}{\begin{eqnarray*}}
\newcommand{\EEAS}{\end{eqnarray*}}
\newcommand{\BEA}{\begin{eqnarray}}
\newcommand{\EEA}{\end{eqnarray}}
\newcommand{\BEQ}{\begin{equation}}
\newcommand{\EEQ}{\end{equation}}
\newcommand{\BIT}{\begin{itemize}}
\newcommand{\EIT}{\end{itemize}}
\newcommand{\BNUM}{\begin{enumerate}}
\newcommand{\ENUM}{\end{enumerate}}

\newcommand{\BA}{\begin{array}}
\newcommand{\EA}{\end{array}}

\newcommand{\QED}{~~\rule[-1pt]{6pt}{6pt}}

\DeclareMathOperator*{\argmin}{arg\,min}
\DeclareMathOperator*{\argmax}{arg\,max}
\DeclareMathOperator*{\minimize}{minimize}

\DeclareMathOperator*{\conv}{conv}
\DeclareMathOperator*{\LMO}{LMO}
\def\xx{{\boldsymbol x}}
\def\bb{{\boldsymbol b}}

\def\dd{{\boldsymbol d}}
\def\ww{{\boldsymbol w}}
\def\yy{{\boldsymbol y}}
\def\vv{{\boldsymbol v}}
\def\rr{{\boldsymbol r}}
\def\ss{{\boldsymbol s}}
\def\zz{{\boldsymbol z}}

\def\defas{\stackrel{\text{def}}{=}}

\def\MM{{\mathcal M}}

\newcommand{\Econd}{\mathbf{E}}

\newcommand{\floor}[1]{\lfloor #1 \rfloor}

\usepackage{times}
\usepackage{graphicx} % more modern
\usepackage{subfigure} 

\usepackage{enumitem}

\usepackage{natbib}

\usepackage[linesnumbered, ruled,vlined]{algorithm2e}
\SetKwBlock{Variantone}{Variant 1}{}
\SetKwBlock{Varianttwo}{Variant 2}{}

\usepackage{hyperref}

\usepackage{empheq}

\icmltitlerunning{Frank-Wolfe with Subsampling Oracle}

\makeatletter
\newcommand{\neutralize}[1]{\expandafter\let\csname c@#1\endcsname\count@}
\makeatother

\newtheorem{thm}{Theorem}

\newenvironment{thmbis}[1]
  {\renewcommand{\thethm}{\ref{#1}$'$}%
   \neutralize{thm}\phantomsection
   \begin{thm}}
  {\end{thm}}
  
\newtheorem{lem}{lemma}

\begin{document} 

%DOUBLE COLUMN PAGE UNCOMMENT
%\twocolumn[
\onecolumn{

\icmltitle{Frank-Wolfe with Subsampling Oracle}
% 
% \icmltitle{An Efficient Algorithm for Interaction Selection}
% \icmltitle{Frank-Wolfe converges over subsampled domains}

% It is OKAY to include author information, even for blind
% submissions: the style file will automatically remove it for you
% unless you've provided the [accepted] option to the icml2017
% package.

% list of affiliations. the first argument should be a (short)
% identifier you will use later to specify author affiliations
% Academic affiliations should list Department, University, City, Region, Country
% Industry affiliations should list Company, City, Region, Country

% you can specify symbols, otherwise they are numbered in order
% ideally, you should not use this facility. affiliations will be numbered
% in order of appearance and this is the preferred way.
\icmlsetsymbol{equal}{*}

\begin{icmlauthorlist}
\icmlauthor{Thomas Kerdreux}{equal,ENS}
\icmlauthor{Fabian Pedregosa}{equal,BERK,ETHZ}
\icmlauthor{Alexandre d'Aspremont}{ENS,CNRS}
\end{icmlauthorlist}

\icmlaffiliation{ENS}{D.I., UMR 8548,
 \'Ecole Normale Sup\'erieure, Paris, France.}
\icmlaffiliation{CNRS}{CNRS, France.}
\icmlaffiliation{BERK}{UC Berkeley, USA.}
\icmlaffiliation{ETHZ}{ETH Zurich, Switzerland.}

\icmlcorrespondingauthor{Thomas Kerdreux}{thomas.kerdreux@inria.fr}
\icmlcorrespondingauthor{Fabian Pedregosa}{f@bianp.net}

% You may provide any keywords that you 
% find helpful for describing your paper; these are used to populate 
% the "keywords" metadata in the PDF but will not be shown in the document
\icmlkeywords{optimization, subsampling, Frank-Wolfe, factorization machines}

\vskip 0.3in

% UNCOMMENT TO HAVE DOUBLE COLUMNS PAPER
%]
}

% this must go after the closing bracket ] following \twocolumn[ ...

% This command actually creates the footnote in the first column
% listing the affiliations and the copyright notice.
% The command takes one argument, which is text to display at the start of the footnote.
% The \icmlEqualContribution command is standard text for equal contribution.
% Remove it (just {}) if you do not need this facility.

%\printAffiliationsAndNotice{}  % leave blank if no need to mention equal contribution
\printAffiliationsAndNotice{\icmlEqualContribution} % otherwise use the standard text.

\begin{abstract}
We analyze two novel randomized variants of the Frank-Wolfe (FW) or conditional gradient algorithm.
While classical FW algorithms require solving a linear minimization problem over the domain at each iteration, the proposed method only requires to solve a linear minimization problem over a small \emph{subset} of the original domain.
The first algorithm that we propose is a randomized variant of the original FW algorithm and achieves a $\mathcal{O}(1/t)$ sublinear convergence rate as in the deterministic counterpart. The second algorithm is a randomized variant of the Away-step FW algorithm, and again as its deterministic counterpart, reaches linear (i.e., exponential) convergence rate making it the first provably convergent randomized variant of Away-step FW. In both cases, while subsampling reduces the convergence rate by a constant factor, the linear minimization step can be a fraction of the cost of that of the deterministic versions, especially when the data is streamed. We illustrate computational gains of the algorithms on regression problems, involving both $\ell_1$ and latent group lasso penalties.
\end{abstract}

\section{Introduction}\label{s:intro}

The Frank-Wolfe (FW) or conditional gradient algorithm~\citep{frank1956algorithm,jaggi2013revisiting} is designed to solve optimization problems of the form
\begin{empheq}[box=\mybluebox]{equation}\tag{OPT}\label{eq:obj_fun}
  \minimize_{\xx \in \MM}\,\vphantom{\sum_i^n} f(\xx) ~, \text{ with $\MM = \conv(\mathcal{A})$}~,
\end{empheq}
where $\mathcal{A}$ is a (possibly infinite) set of vectors which we call \emph{atoms}, and $\conv(\mathcal{A})$ is its convex hull. The FW algorithm and variants have seen an impressive revival in recent years, due to their low memory requirements and projection-free iterations, which make them particularly appropriate to solve large scale convex problems, for instance convex relaxations of problems written over combinatorial polytopes~\citep{zaslavskiy2009path,joulin2014efficient,vogelstein2015fast}.

The Frank-Wolfe algorithm is projection-free, i.e. unlike most methods to solve \eqref{eq:obj_fun}, it does not require to compute a projection onto the feasible set $\mathcal{M}$. Instead, it relies on a linear minimization oracle over a set $\mathcal{A}$, written $\text{LMO}(\cdot, \mathcal{A})$, which solves the following linear problem
\begin{equation}
\text{LMO}(\rr, \mathcal{A}) \in \argmin_{\vv \in \mathcal{A}}\langle \vv, \rr\rangle~.
\end{equation}
For some constraint sets, such as the nuclear or latent group norm ball \cite{jaggi2010simple,vinyes2017fast}, computing the LMO can be orders of magnitude faster than projecting. 
Another feature of FW that has greatly contributed to its practical success is its low memory requirements. The algorithm maintains its iterates as a convex combination of a few atoms, enabling the resulting sparse and low rank iterates to be stored efficiently. This feature allows the FW algorithm to be used in situations with a huge or even infinite number of features, such as architecture optimization in neural networks~\citep{ping2016learning} or estimation of an infinite-dimensional sparse matrix arising in multi-output polynomial networks~\citep{NIPS2017_6927}.
% Recently, this feature has allowed the FW algorithm to optimize on the architecture of neural networks \citep{ping2016learning}.

% FW algorithm was studied by \cite{garber2016linear}. It was also used in \cite{yurtsever2017sketchy} with the Trace Norm polytope and finally appeared as a way to optimize on the architecture of neural networks \citep{ping2016learning,scardapane2017group}. In such situations, the optimization is done over the node as well as the link between them which results in computational overloads.

Despite these attractive properties, for problems with a large number of variables or with a very large atomic set (or both), computing the full gradient and LMO at each iteration can become prohibitive. Designing variants of the FW algorithm which alleviate this computational burden would have a significant practical impact on performance.

One recent direction to achieve this is to replace the LMO with a \emph{randomized} linear oracle in which the linear minimization is performed only over a random sample of the original atomic domain. This approach has proven to be highly successful on specific problems such as structured SVMs~\citep{lacoste2012block} and $\ell_1$-constrained regression \citep{Frandi2016}, however little is known in the general case. Is it possible to design a FW variant with a randomized oracle that achieves the same convergence rate (up to a constant factor) as the non-randomized variant? Can this be extended to linearly-convergent FW algorithms~\citep{lacoste2013affine,lacoste2015global,garber2015faster}? In this paper we give a positive answer to both questions and explore the trade-offs between subsampling and convergence rate.
% Most of the time it \emph{is} possible to get away with subsampling.

%Our contribution
\paragraph{Outline and main contribution.} 
The main contribution of this paper is to develop and analyze two algorithms that share the low memory requirements and projection-free iterations of FW, but in which the LMO is computed only over a random subset of the original domain. 
In many cases, this results in large computational gains in computing the LMO which can also speed up the overall FW algorithm. In practice, the algorithm will run a larger number of cheaper iterations, which is typically more efficient for very large data sets (e.g. in a streaming model where the data does not fit in core memory and can only be accessed by chunks).
% We highlight technical issues in previous work and provide theoretical convergence guarantee for both algorithms, showing that they reach the same asymptotic convergence rate as their deterministic counter-part.
The paper is structured as follows
\begin{itemize}[leftmargin=*]
\item \S\ref{s:fw} describes the Randomized FW algorithm, proving a sublinear convergence rate.
\item \S\ref{s:rafw} describes ``Randomized Away FW'', a variant of the above algorithm with linear convergence rate on polytopes. To the best of our knowledge this is the first provably convergent randomized version of the Away-steps FW algorithm.
\item Finally, in \S\ref{s:Numerical_Results} we discuss implementation aspects of the proposed algorithms and study their performance on  lasso and latent group lasso problems.
\end{itemize}

Note that with the proven sub-linear rate of convergence for Randomized FW, the cost of the LMO is reduced by the subsampling rate, but this is compensated by the fact that the number of iteration required by RFW to reach same convergence guarantee as FW is itself multiplied by the sampling rate. Similarly the linear convergence rate in Randomized AFW does not theoretically show a computational advantage since the number of iterations is multiplied by the squared sampling rate, in our highly conservative bounds at least. Nevertheless, our numerical experiments show that randomized versions are often numerically superior to their deterministic counterparts.

\subsection{Related work} 

%We briefly discuss different approaches that have been proposed with the goal of alleviating the cost of the linear minimization oracle.
Several references have focused on reducing the cost of computing the linear oracle. The analysis of \citep{jaggi2013revisiting}
allows for an error term in the LMO, and so a randomized linear oracle could in principle be analyzed under this framework. However, this is not fully satisfactory as it requires the approximation error to decrease towards zero as the algorithm progresses. In our algorithm, the subsampling approximation error doesn't need to decrease.

\citet{lacoste2012block} studied a randomized FW variant named block-coordinate FW in which at each step the LMO is computed only over a subset (block) of variables. In this case, the approximation error need not decrease to zero, but the method can only be applied to a restricted class of problems: those with block-separable domain, leaving out important cases such as $\ell_1$--constrained minimization.
%This approach was recently revisited by \citet{osokin2016minding} to consider different improvements such as away-step FW.
Because of the block separability, a more aggressive step-size strategy can be used in this case, resulting overall in a different algorithm.
%{\blue TK: I suggest that we say something like  block coordinate in other framework (proximal/projected gradient descent) are very successful but that can be adapted in FW scheme's only when the polytope are separable. Here we propose to extend further that properties and that it directly induce a block coordinate LMO (because that is the flaw in our argument: subsampling does not generally means to do block coordinate descent). }

Finally, \citet{frandi2014complexity} proposed a FW variant which can be seen as a special case of our Algorithm~\ref{algo:Randomized_FW_general} for the Lasso problem, analyzed in \cite{Frandi2016}. Our analysis here brings three key improvements on this last result. First, it is provably convergent for arbitrary atomic domains,  not just the $\ell_1$ ball (furthermore the proof in \cite{Frandi2016} has technical issues discussed in \ref{s:appendix_incomplete_frandi}). Second, it allows a choice of step size that does not require exact line-search (Variant 2), which is typically only feasible for quadratic loss functions. Third, we extend our analysis to linearly-convergent FW variants such as the Away-step FW.

% In this work we extend this algorithm to arbitrary domains (Algorithm
% For one regularization parameter {\blue (I am trying to say that they compared the result when solving the lasso with just one parameter $\lambda$ which is somehow unfair...)}, experimental analysis exhibited state of the art performance. Though we found their theoretical analysis incomplete (see ).

A different technique to alleviate the cost of the linear oracle was recently proposed by \citet{braun2017lazifying}. In that work, the authors propose a FW variant that replaces the LMO by a ``weak'' separation oracle and showed significant speedups in wall-clock performance on problems such as the video co-localization. This approach was combined with gradient sliding in \cite{lan2017conditional}, a technique \citep{lan2016conditional} that allows to skip the computation of gradients from time to time. 
However, for problems such as Lasso or latent group lasso, a randomized LMO avoids all full gradient computations, while the lazy weak separation oracle still requires it. Combining these various techniques is an interesting open question.
% Both techniques can be seen as orthogonal to our method, and are beneficial in different circumstances. For instance, in the case of the 

%Why go beyond Frandi
% The classic FW algorithm does not give the best ratio between the size of the support and the number of iteration. Accentuated by the fact that sub-sampling the polytope results in less relevant choice of atom (for the sake of computation), we propose a randomized version of the AFW (RAFW), more adapted to ensure low cost memory.

Proximal coordinate-descent methods~\citep{richtarik2014iteration} (not based on FW) have also been used to solve problems with a huge number of variables. They are particularly effective when combined with variable screening rules such as~\citep{strongrules,fercoq2015mind}. However, for constrained problems they require evaluating a projection operator, which on some sets such as the latent group lasso ball can be much more expensive than the LMO. Furthermore, these methods require that the projection operator is block-separable, while our method does not.

\paragraph{Notation.}
We denote vectors with boldface lower case letters (i.e., $\xx$), and sets in calligraphic letter (i.e., $\mathcal{A}$).
We denote $\text{clip}_{[0,1]}(s)=\max\{0,\min\{1,s\}\}$. Probability is denoted $\mathcal{P}$. The cardinality of a set $\mathcal{A}$ is denoted $|\mathcal{A}|$. For $\xx^*$ a solution of \eqref{eq:obj_fun}, we denote $h(\xx) = f(\xx) -f(\xx^*)$. %For an atomic set $\mathcal{A}$ and a particular convex combination $\xx = \sum_{\vv\in\mathcal{A}}{\alpha_{\vv}\vv}$, we consider its support  $\mathcal{S}(\xx)=\{\vv\in\mathcal{A}\text{ s.t. }\alpha_{\vv}>0\}$. For an iterate $\xx_t$ of a FW type algorithm, we name $\mathcal{S}(\xx_t)$ by $\mathcal{S}_t$.

\emph{Randomized vs stochastic.} We denote FW variants with randomness in the LMO \emph{randomized} and reserve the name \emph{stochastic} for FW variants that replace the gradient with a stochastic approximation, as in \citep{hazan2016variance}.

\section{Randomized Frank-Wolfe}\label{s:fw}

In this section we present our first contribution, a FW variant that we name Randomized Frank-Wolfe (RFW). The method is detailed in Algorithm \ref{algo:Randomized_FW_general}.
Compared to the standard FW algorithm, it has the following two distinct features.

First,
 the LMO is computed over a random subset $\mathcal{A}_t \subseteq \mathcal{A}$ of the original atomic set in which each atom is equally likely to appear, i.e., in which  $\mathcal{P}(\vv \in \mathcal{A}_t)\!=\!\eta$ for all ${\vv \in \mathcal{A}}$ (Line~\ref{line:subsample}). For discrete sets this can be implemented simply by drawing uniformly at random a fixed number of elements at each iteration. The sampling parameter $\eta$ controls the fraction of the domain that is considered by the LMO at each iteration. If $\eta = 1$, the LMO considers the full domain at each iteration and the algorithm defaults to the classical FW algorithm. However, for $\eta < 1$, the LMO only needs to consider a fraction of the atoms in the original dataset and can be faster than the FW LMO.

Second, because of this subsampling we can no longer guarantee that the atom chosen by the LMO is a descent direction and so it is no longer possible to use the ``oblivious'' (i.e., independent on the result of the LMO) $2/(2+t)$ step-size commonly  used in the FW algorithm.
We provide two possible choices for this step-size: the first variant (Line~\ref{alg:var0_1}) chooses the step-size by exact line search and requires to solve a 1-dimensional convex optimization problem. This approach is efficient when this sub-problem has a closed form solution, as it happens for example in the case of quadratic loss functions. The second variant does not need to solve this sub-problem, but in exchange requires to have an estimate of the curvature constant $C_f$ (defined in next subsection). Note that in absence of an estimate of this quantity, one can use the bound $C_f \leq \text{diam}(\mathcal{M})^2 L$, where $L$ is the Lipschitz constant of $\nabla f$ and $\text{diam}(\mathcal{M})$ is the diameter of the domain in euclidean norm.

\paragraph{Gradient coordinate subsampling.}
We note that the gradient of $f$ only enters Algorithm~\ref{algo:Randomized_FW_general} through the computation of the randomized LMO, and so only the dot product between the gradient and the subsampled atomic set are truly necessary.
In some cases the elements of the atomic set have a specific structure that makes computing dot products particularly effective. For example, when the atomic elements are sparse, only the coordinates of the gradient that are in the support of the atomic set need to be evaluated.
As a result, for
sparse atomic sets such as the $\ell_1$ ball, the group lasso ball (also known as $\ell_1/\ell_2$ ball), or even the latent group lasso~\citep{obozinski2011group}  ball,
 only a few coordinates of the gradient need to be evaluated at each iteration. The number of exact gradients that need to be evaluated will depend on both the sparsity of this atomic set and the subsampling rate. For example, in the case of the $\ell_1$ ball, the extreme atoms have a single nonzero coefficient, and so RFW only needs to compute on average $d\eta$ gradient coefficients at each iteration, where $d$ denotes the ambient dimension.

\setlength{\textfloatsep}{10pt}% Less vertical space after algorithm
\begin{algorithm}[t]
  % \begin{algorithmic}[1]
  {\bfseries Input:}  $\xx_0 \in \mathcal{M}$, sampling ratio $0 < \eta \leq 1$.
  \For{$t=0, 1 \ldots, T $}{
      Choose $\mathcal{A}_t$ such that $\mathcal{P}(\vv\!\in\!\mathcal{A}_t)\!=\!\eta$ for all $\vv \in \mathcal{A}$\label{line:subsample}

      $\ss_{t} = \LMO(\nabla f(\xx_t), \mathcal{A}_t)$

      \Variantone(\hfill $\triangleright$ set $\gamma_t$ by line-search){
      $\gamma_t = \argmax_{\gamma \in [0, 1]} f((1 - \gamma_t)\xx_{t} + \gamma_t \ss_{t})$\label{alg:var0_1}
      }

      \Varianttwo{$\gamma_t=\text{clip}_{[0,1]}(\langle -\nabla f(\xx_{t}), \ss_t-\xx_t\rangle/{C_f})$
      }
      $\xx_{t+1} = (1 - \gamma_t)\xx_{t} + \gamma_t \ss_{t}$
    }
  % \end{algorithmic}
    \caption{Randomized Frank-Wolfe algorithm}\label{algo:Randomized_FW_general}
\end{algorithm}

%\paragraph{Lost property.}
%Because the linear minimization is done on a subset of the constraint set, the randomized LMO (Line \ref{line:subsample} in Algorithm \ref{algo:Randomized_FW_general}) does not yield anymore a dual gap guarantee. To stop the algorithm, a safe and computationally saving way is to perform a full LMO every $k\floor{\frac{1}{\eta}}$, with $k\in\mathbb{N}^*$.

\paragraph{Stopping criterion.} A side-effect of subsampling the linear oracle is that $\langle -\nabla f(\xx_t);\ss_t-\xx_t\rangle$, where $\ss_t$ is the atom selected by the randomized linear oracle is not, unlike in the non-randomized algorithm, an upper bound on $f(\xx_t)-f(\xx^*)$. This property is a feature of FW algorithms that cannot be retrieved in our variant. As a replacement, the stopping criteria that we propose is to compute a full LMO every $k\floor{\frac{1}{\eta}}$ iterations, with $k\in\mathbb{N}^*$ ($k=2$ is a good default value).

\subsection{Analysis}
In this subsection we prove an $\mathcal{O}(1/t)$ convergence rate for the RFW algorithm. As is often the case for FW-related algorithms, our convergence result will be stated in terms of the \emph{curvature constant} $C_f$, which is defined as follows for a convex and differentiable function $f$ and a convex and compact domain $\mathcal{M}$:
%
%
% The  \emph{curvature constant} of $f$ is a measure of deviation from linearity defined as follows:
% % As for other FW variants, our analysis will depend on the  \emph{curvature constant} of $f$, which is defined as follows:
% . , this is defined as
\begin{equation*}
~C_f\defas\!\!\!\!\underset{\substack{\xx,\ss\in\mathcal{M},\gamma\in [0,1] \\ \yy=\xx+\gamma (\ss-\xx)}}{\text{sup }} {\frac{2}{\gamma^2}\big( f(\yy)-f(\xx)-\langle\nabla f(\xx),\yy-\xx\rangle \big)}.
\end{equation*}
It is worth mentioning that a bounded curvature constant $C_f$ corresponds to a Lipschitz assumption on the gradient of $f$~\citep{jaggi2013revisiting}.

\begin{theorem}\label{th:expectation_rate} Let $f$ be a function with bounded smoothness constant $C_f$ and subsampling parameter $\eta\in(0,1]$. Then
Algorithm \ref{algo:Randomized_FW_general} (in both variants) converges towards a solution of~\eqref{eq:obj_fun}. Furthermore, the following inequality is satisfied:
\begin{equation}
\mathbb{E}[h(\xx_{T})]\leq \frac{2(C_f+f(\xx_0)-f(\xx^*))}{\eta T+2}~,
\end{equation}
where $\mathbb{E}$ is a full expectation over all randomness until iteration $T$.
\end{theorem}

\begin{proof}
See \ref{apx:rfw_proof}.
\end{proof}

The rate obtained in the previous theorem is similar to known bounds for FW. For example, \citep[Theorem 1]{jaggi2013revisiting} established for FW a bound of the form
\begin{eqnarray}
h(\xx_T)\leq \frac{2C_f}{T+2}~.
\end{eqnarray}
This is similar to the rate of Theorem \ref{th:expectation_rate}, except for the factor $\eta$ in the denominator. Hence, if our updates are $\eta$ times as costly as the full FW update (as is the case e.g. for the $\ell_1$ ball), then the theoretical convergence rate is the same.
% From the theoretical point of view, there does not seem to be a clear advantage of RFW. Furthermore,
This bound is likely tight, as in the worst case one will need to sample the whole atomic set to decrease the objective if there is only one descent direction. This is however a very pessimistic scenario, and in practice good descent directions can often be found without sampling the whole atomic set.
As we will see in the experimental section, despite these conservative bounds, the algorithm often exhibits large computational gains with respect to the deterministic algorithm.

\section{Randomized Away-steps Frank-Wolfe}\label{s:rafw}

% A popular variant of the FW algorithm is the Away-steps FW variant of \citet{guelat1986some}.
% This algorithm can be applied in the case of a polytope domain and adds the option to move away from an atom in the current representation of the iterate. It is hence slightly less general than the FW algorithm, as it only applies to domains with a finite atomic representation, but it was shown recently to have much better convergence properties, such as linear (i.e. exponential) convergence rates for generally-strongly convex objectives \citep{garber2013linearly,beck2013convergence,lacoste2015global}.

A popular variant of the FW algorithm is the Away-steps FW variant of \citet{guelat1986some}. This algorithm adds the option to move away from an atom in the current representation of the iterate. In the case of a polytope domain, it was recently shown to have much better convergence properties, such as linear (i.e. exponential) convergence rates for generally-strongly convex objectives \citep{garber2013linearly,beck2013convergence,lacoste2015global}.

In this section we describe the first provably convergent randomized version of the Away-steps FW, which we name \emph{Randomized Away-steps FW} (RAFW). We will assume throughout this section that the domain is a polytope, i.e. that $\mathcal{M} = \conv(\mathcal{A})$, where $\mathcal{A}$ is a finite set of atoms. We will make use of the following notation.
\begin{itemize}[leftmargin=*]
\item \emph{Active set.} We denote by $\mathcal{S}_t$ the active set of the current iterate, i.e. $\xx_t$ decomposes as $\xx_t=\sum_{\vv\in\mathcal{S}_t}{\alpha^{(t)}_{\vv}\vv}$, where $\alpha^{(t)}_{\vv} > 0$ are positive weights that are iteratively updated.

\item \emph{Subsampling parameter.} %The method depends on a subsampling parameter $p$ that controls the amount of computation per iteration of the LMO. Since in this case the atomic set is finite, $p$ denotes an integer $1 \leq p \leq |\mathcal{A}|$, and is equivalent to $\eta |\mathcal{A}|$ in the RFW formulation of \S\ref{s:fw}.
The method depends on a subsampling parameter $p$. It controls the amount of computation per iteration of the LMO. In this case, the atomic set is finite and $p$ denotes an integer $1 \leq p \leq |\mathcal{A}|$. This sampling rate is approximately $\floor{\eta |\mathcal{A}|}$ in the RFW formulation of \S\ref{s:fw}.
\end{itemize}

The method is described in Algorithm~\ref{algo:Randomized_away_general} and, as in the Away-steps FW, requires computing two linear minimization oracles at each iteration. Unlike the deterministic version, the first oracle is computed on the subsampled set ${\mathcal{S}_t \cup \mathcal{A}_t}$ (Line~\ref{line:rafw_subsampled_lmo}), where $\mathcal{A}_t$ is a subset of size 
${\min\{p, |\mathcal{A}\!\setminus\!\mathcal{S}_t|\}}$, sampled uniformly at random from  $\mathcal{A}\setminus\!\mathcal{S}_t$. The second LMO (Line \ref{line:away_lmo}) is computed on the active set, which is also typically much smaller than the atomic domain.

As a result of both oracle calls, we obtain two potential descent directions, the RFW direction $\dd_t^{\text{FW}}$ and the Away direction $\dd_t^{\text{A}}$. The chosen direction is the one that correlates the most with the negative gradient, and a maximum step size is chosen to guarantee that the iterates remain feasible (Lines \ref{line:decision_step}--\ref{l:gammamax2}).

\paragraph{Updating the support.} 

Line \ref{line:RAFW_update_x} requires updating the support and the associated $\alpha$ coefficients.
For a FW step we have $\mathcal{S}_{t+1}=\{\ss_t\}$ if $\gamma_t=1$  and otherwise $\mathcal{S}_{t+1}=\mathcal{S}_{t}\cup\{\ss_t\}$. The corresponding update of the weights is $\alpha^{(t+1)}_{\vv}=(1-\gamma_t)\alpha^{(t)}_{\vv}$ when $\vv\in\mathcal{S}_t\setminus\{\ss_t\}$ and $\alpha^{(t+1)}_{\ss_t}=(1-\gamma_t)\alpha^{(t)}_{\ss_t}+\gamma_t$ otherwise.

\begin{algorithm}[t]
  \KwIn{$\xx_{0} \in \mathcal{M}$, $\xx_0 = \sum_{\vv\in\mathcal{A}}{\alpha^{(0)}_{\vv}\vv}$ with $|\mathcal{S}_0|=s$, a subsampling parameter $1 \leq p\leq |\mathcal{A}|$.}
  \For{$t=0, 1 \ldots, T $}{
      %$\mathcal{A}_{t} \gets $ draw $p$ elements with replacement from ${\mathcal{A}\setminus \mathcal{S}_{t}}$.\label{line:sub_sampling}
      
      Get $\mathcal{A}_{t}$ by sampling $\min\{p,\!|\mathcal{A}\!\setminus\!\mathcal{S}_t|\}$ elements uniformly  from  ${\mathcal{A}\!\setminus\!\mathcal{S}_{t}}$.\label{line:sub_sampling}
      
      % $\mathcal{A}_{t} = $ \label{line:def_A_t}
      
      Compute $\ss_{t} = \LMO(\nabla f(\xx_t), \mathcal{S}_{t}\cup \mathcal{A}_{t})$\label{line:rafw_subsampled_lmo}
      
      Let $\dd_{t}^{\text{FW}} = \boldsymbol{s}_{t}-\xx_{t}$\hfill $\triangleright$ \text{RFW direction}\label{direction:sub_FW}
      
      Compute $\boldsymbol{v}_{t} = \LMO(-\nabla f(\xx_t), \mathcal{S}_{t})$ \label{line:away_lmo}

      Let $\dd_{t}^{A}=\xx_{t}-\vv_{t}$.\label{direction:away_FW}\hfill $\triangleright$ \text{Away direction}
      
      \eIf{$\langle - \nabla f(\xx_t) , \dd_{t}^{\text{FW}} \rangle \geq \langle - \nabla f(\xx_t) ,\dd_{t}^{A}\rangle$ \label{line:decision_step}}{
      
      $\dd_t=\dd_{t}^{\text{FW}}$ and $\gamma_{\text{max}}=1$\label{l:gammamax} \hfill $\triangleright$ \text{FW step}
      }{
      $\dd_t=\dd_{t}^{A}$ and $\gamma_{\text{max}}\!=\!{\alpha^{(t)}_{\vv_t}}/{(1\!-\!\alpha^{(t)}_{\vv_t})}$\label{l:gammamax2} \hfill \!\!\!$\triangleright$ \text{Away step}
      }
      
      Set $\gamma_t$ by line-search, with
      \qquad $\gamma_t = \argmax_{\gamma \in [0, \gamma_{\text{max}}]} f((1 - \gamma_t)\xx_{t} + \gamma_t)$

      Let $\xx_{t+1} = \xx_{t} + \gamma_t \dd_t$\hfill $\triangleright$ update $\alpha^{(t+1)}\!$ (see text)\label{line:RAFW_update_x}
      
      Let $\mathcal{S}_{t+1} = \{\vv \in \mathcal{A} \text{ s.t. } \alpha_\vv^{(t+1)} > 0\}$ 
    }
    % \textbf{return}: $\xx^{k+1}$
  \caption{Randomized Away-steps FW (RAFW)} \label{algo:Randomized_away_general}
\end{algorithm}

For an away step we instead have the following update rule.
When $\gamma_t=\gamma_{\text{max}}$ (which is called a \textit{drop step}), then 
 $\mathcal{S}_{t+1}=\mathcal{S}_{t}\setminus\{\vv_t\}$. Combined with $\gamma_{\text{max}}<1$ (or equivalently $\alpha_{\vv_t}\leq \frac{1}{2}$) we call them \textit{bad drop step}, as it corresponds to a situation in which we are not able to guarantee a geometrical decrease of the dual gap.
 
For away steps in which $\gamma_t < \gamma_{\text{max}}$, the away atom is not removed from the current representation of the iterate. Hence $\mathcal{S}_{t+1}=\mathcal{S}_{t}$, $\alpha^{(t+1)}_{\vv}={(1+\gamma_t)\alpha^{(t)}_{\vv}}$ for $\vv\in\mathcal{S}_{t}\setminus \{\vv_t\}$ and $\alpha^{(t+1)}_{\vv_t}={(1+\gamma_t)}\alpha^{(t)}_{\vv_t}-\gamma_t$ otherwise.
 %When $\gamma_{\text{max}}<1$ (or equivalently $\alpha_{\vv_t}\leq \frac{1}{2}$) (which we call \textit{bad drop step}), then $\mathcal{S}_{t+1}=\mathcal{S}_{t}$ and $\alpha^{(t+1)}_{\vv}={(1+\gamma_t)\alpha^{(t)}_{\vv}}$ for $\vv\in\mathcal{S}_{t}\setminus \{\vv_t\}$ and $\alpha^{(t+1)}_{\vv_t}={(1+\gamma_t)}\alpha^{(t)}_{\vv_t}-\gamma_t$ otherwise.
 %The name bad drop step comes from the fact that in this case we are not able to guarantee a geometrical decrease of the dual gap. Fortunately,  we will be able to upper-bound the number of drop steps and consequently that of bad drop steps. 

Note that when choosing Away step in Line \ref{l:gammamax2}, it cannot happen that $\alpha_{\vv_t} = 1$. Indeed this would imply $\xx_t = \vv_t$, and so $\dd_t^A = 0$. % Because $\vv_t\in \mathcal{S}_t\cup\mathcal{A}_t$, necessarily $\langle - \nabla f(\xx_t) , \dd_{t}^{\text{FW}} \rangle \geq 0$ and hence a FW choice step.
Since we would have $\mathcal{S}_t=\{\vv_t\}$ and the LMO of Line \ref{line:rafw_subsampled_lmo} is performed over $\mathcal{S}_t\cup\mathcal{A}_t$, we necessarily have $\langle - \nabla f(\xx_t) , \dd_{t}^{\text{FW}} \rangle \geq 0$. It thus leads to a choice of FW step, contradiction.

\paragraph{Per iteration cost.} Establishing the per iteration cost of this algorithm is not as straightforward as for RFW, as the cost of some operations depends on the size of the active set, which varies throughout the iterations. However, for problems with sparse solutions, we have observed empirically that the size of the active set remains small, making the cost of the second LMO and the comparison of Line \ref{line:decision_step} negligible compared to the cost of an LMO over the full atomic domain. In this regime, and assuming that the atomic domain has a sparse structure that allows gradient coordinate subsampling, RAFW can achieve a per iteration cost that is, like RFW, roughly $|\mathcal{A}|/p$ times lower than that of its deterministic counterpart.

\subsection{Analysis}

We now provide a convergence analysis of the Randomized Away-steps FW algorithm.
These convergence results are stated in terms of the away curvature constant $C_f^A$ and the geometric strong convexity $\mu_f^A$, which are described in \ref{apx:linear_proof} and in~\citep{lacoste2015global}. Throughout this section we assume that $f$ has bounded $C_f^A$ (note that the usual assumption of Lipschitz continuity of the gradient over compact domain implies this) and strictly positive geometric strong convexity constant $\mu_f^A$.

\begin{theorem}\label{th:RAFW_expectation_result}  Consider the set $\mathcal{M}=\conv(\mathcal{A})$, with $\mathcal{A}$ a finite set of extreme atoms, after $T$ iterations of Algorithm~\ref{algo:Randomized_away_general} (RAFW) we have the following linear convergence rate
\begin{eqnarray}\label{eq:RAFW_bound_result}
\mathbb{E}\big[ h(\xx_{T+1})\big]\leq \big(1- \eta^2 \rho_f\big)^{\max\{0,\floor{{(T-s)}/{2}}\}} h(\xx_{0})~,
\end{eqnarray}
with $\rho_f=\frac{\mu_f^A}{4C_f^A}$, $\eta=\frac{p}{|\mathcal{A}|}$ and $s=|\mathcal{S}_0|$.
\end{theorem}

\begin{proof}
See \ref{apx:linear_proof}.
\end{proof}

\begin{sketchproof}
Our proof structure roughly follows that of the deterministic case in \citep{lacoste2015global,beck2013convergence} with some key differences due to the LMO randomness, and can be decomposed into three parts. 

The \emph{first part} consists in upper bounding $h_t$ and is no different from the proof of its deterministic counterpart \citep{lacoste2015global,beck2013convergence}.

The \emph{second part} consists in lower bounding the progress $h_t - h_{t+1}$. For this algorithm we can guarantee a decrease of the form 
\begin{eqnarray}\label{eq:POC}
h_{t+1}\leq h_t \big( 1-\rho_f \big(\frac{g_t}{\tilde{g}_t}\big)^2  \big)^{z_t}~,
\end{eqnarray}
where $g_t = \langle -\nabla f(\xx_t), \ss_t - \vv_t\rangle$ is the \textit{partial pair-wise dual gap} while $\tilde{g}_t$ is the \textit{pair-wise dual gap}, in which $\ss_t$ is replaced by the result of a full (and not subsampled) LMO.

We can guarantee a possible geometric decrease on $h_t$ at each iteration, except for bad drop steps,  where we can only secure $h_{t+1}\leq h_t$. We mark these by setting $z_t=0$.

One crucial issue is then to quantify ${g_t}/{\tilde{g}_t}$. This can be seen as a measure of the quality of the subsampled oracle: if it selects the same atom as the non-subsampled oracle the quotient will be 1, in all other cases it will be $\leq 1$.% We will lower bound the probability of $g_t = \tilde{g}_t$ in Lemma \ref{lemma:proba_conditional_non_drop_step}, and lower bound the quotient by $0$ in the other cases.

%What is left now is to bound the number of bad drop steps (i.e., where $z_t=0$). {\blue Phrase on how you do this.}

To ensure a geometrical decrease we further study the probability of events $z_t=1$ and $\widetilde g_t = g_t$: first, we produce a simple bound on the number of bad drop steps (where $z_t=0$). Second, when $z_t=1$ holds, Lemma \ref{lemma:proba_conditional_non_drop_step} provides a lower bound on the probability of $g_t=\widetilde g_t$.

The \emph{third and last part} of the proof analyzes the expectation of the decrease rate $\prod_{t=0}^{T}{(1-\rho_f \big(\frac{g_t}{\tilde{g}_t}\big)^2 )^{\zz_t}}$ given the above discussion. We produce a conservative bound assuming the maximum possible number of bad drop steps. The key element in this part is to make this maximum a function of the size of the support of the initial iterate and of the number of iteration. The convergence bound is then proven by induction. \end{sketchproof}

\paragraph{Comparison with deterministic convergence rates.} The rate for away Frank-Wolfe in \citep[Theorem 8]{lacoste2015global}, after $T$ iteration is 
\begin{eqnarray}
 h(\xx_{T+1})\leq \big(1- \rho_f\big)^{\floor{{T}/{2}}} h(\xx_{0})~.
\end{eqnarray}
Due to the dependency on $\eta^2$ of the convergence rate in Theorem~\ref{th:RAFW_expectation_result}, our bound does not show that RAFW is computationally more efficient than AFW. Indeed we use a very conservative proof technique in which we measure progress only when the sub-sampling oracle equals the full one. Also, the cost of both LMOs depends on the support of the iterates which is unknown a priori except for a coarse upper bound (e.g. the support cannot be more than the number of iterations). Nevertheless, the numerical results do show speed ups compared to the deterministic method.

\paragraph{Beyond strong convexity.}
The strongly convex objective assumption may not hold for many problem instances. However, the linear rate easily holds for $f$ of the form $g(\boldsymbol{A}\xx)$ where $g$ is strongly convex and $\boldsymbol{A}$ a linear operator. This type of function is commonly know as a $\tilde\mu$-generally strongly convex function \cite{beck2013convergence,wang2014iteration} or \cite{lacoste2015global} (see ``Away curvature and geometric strong convexity'' in \ref{apx:linear_proof} for definition). The proof simply adapts that of \citep[Th. 11]{lacoste2015global} to our setting.

\begin{theorem}\label{th:Away_CV_Generalized_Strongly_Convex}  Suppose $f$ has bounded smoothness constant $C_f^A$ and is  $\tilde{\mu}$-generally-strongly convex.  Consider the set $\mathcal{M}=\text{conv}(\mathcal{A})$, with $\mathcal{A}$ a finite set of extreme atoms. Then after $T$ iterations of Algorithm \ref{algo:Randomized_away_general}, with $s=|\mathcal{S}_0|$ and a $p$ parameter of sub-sampling, we have
\begin{eqnarray}\label{eq:RAFW_convergence_rate_generally_strongly_convex}
\mathbb{E}\big[ h(\xx_{T+1})\big]\leq \big(1- \eta^2 \tilde{\rho}_f\big)^{\max\{0,\floor{\frac{T-s}{2}}\}} h(\xx_{0})~,
\end{eqnarray}
with $\tilde{\rho}_f=\frac{ \tilde{\mu}}{4 C_f^A}$ and $\eta=\frac{p}{|\mathcal{A}|}$.
\end{theorem}

\begin{proof}
See end of \ref{apx:linear_proof}.
\end{proof}

\section{Applications}\label{s:Numerical_Results}
%in which situations did we compared...
In this section we compare the proposed methods with their deterministic versions. We consider two regularized least squares problems: one with $\ell_1$ regularization and another one with latent group lasso (LGL) regularization. In the first case, the domain is a polytope and as such the analysis of AFW and RAFW holds.

%Nbr of iterations and 
We will display the FW gap versus number of iterations, and also cumulative number of  computed gradient coefficients, which we will label \textit{``nbr coefficients of grad''}. This allows to better reflect the true complexity of our experiments since sub-sampling the LMO in the problems we consider amounts to computing the gradient on a batch of coordinates.

%wall-clock comparison
In the case of latent group lasso, we also compared the performance of RFW against FW in terms of wall-clock time on a large dataset stored in disk and accessed sequentially in chunks (i.e. in streaming model). %In this setting we perform a wall-clock time comparison of both methods. 

\subsection{Lasso problem}

\paragraph{Synthetic dataset.}
We generate a synthetic dataset following the setting of \cite{lacoste2015global},  with a Gaussian design matrix $A$ of size $(200,500)$ and noisy measurements $\bb = A\xx^* + \boldsymbol\varepsilon$, with $\boldsymbol{\varepsilon}$ a random Gaussian vector and $\xx^*$ a vector with $10\%$ of nonzero coefficients and values in $\{-1,+1\}$.

In Figures \ref{fig:RFW_FW_Lasso} and \ref{fig:RAFW_AFW_Lasso}, we consider a problem of the form \eqref{eq:obj_fun}, where the domain is an $\ell_1$ ball, a problem often referred to as Lasso. We compare FW against RFW, and AFW against RAFW. The $\ell_1$ ball radius set to $40$, so that the unconstrained optimum lies outside the domain.

\paragraph{RFW experiments.}
Figure \ref{fig:RFW_FW_Lasso} compares FW and RFW. Each call to the randomized LMO outputs a direction, likely less aligned with the opposite of the gradient than the direction proposed by FW, which explains why RFW requires more iterations to converge on the upper left graph of Figure \ref{fig:RFW_FW_Lasso}. Each call of the randomized LMO is cheaper than the LMO in terms of number of computed coefficients of the gradient, and the trade-off is beneficial as can be seen on the bottom left graph, where RFW outperforms its deterministic variant in terms of \textit{nbr coefficients of grad}.

Finally, the right panels of Figure \ref{fig:RFW_FW_Lasso} provide an insight on the evolution of the sparsity of the iterate, depending on the algorithm. FW and RFW perform similarly in terms of the fraction of recovered support (bottom right graph). In terms of the sparsity of the iterate, RFW under-performs FW (upper right graph). This can be explained as follows: because of the sub-sampling, each atom of the randomized LMO provides a direction less aligned with the opposite of the gradient than the one provided by the LMO. Each update in such a direction may result in putting weight on an atom that would better be off the representation of the iterate. It impacts the iterate all along the algorithm as vanilla FW removes past atoms from the representation only by multiplicatively shrinking their weight.

% each atom of the randomized LMO will be less likely to be as good as that of the LMO. Such choices impact the sparsity of the iterate all along the algorithm as vanilla FW does not allow us to directly remove past atoms from the representation (a problem solved by AFW below). 

\begin{figure}
\begin{minipage}{0.99\linewidth}
\includegraphics[width=\linewidth]{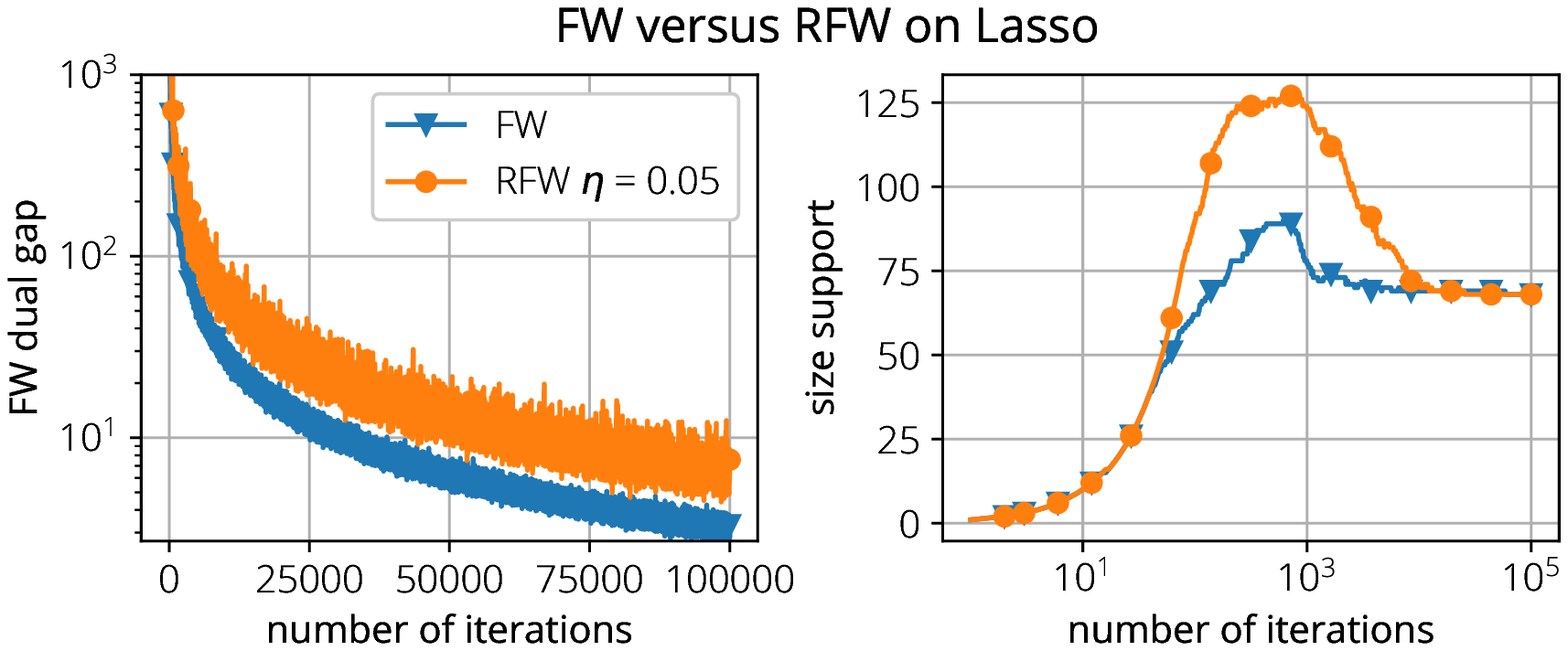}
\end{minipage}

\begin{minipage}{0.99\linewidth}
\includegraphics[width=\linewidth]{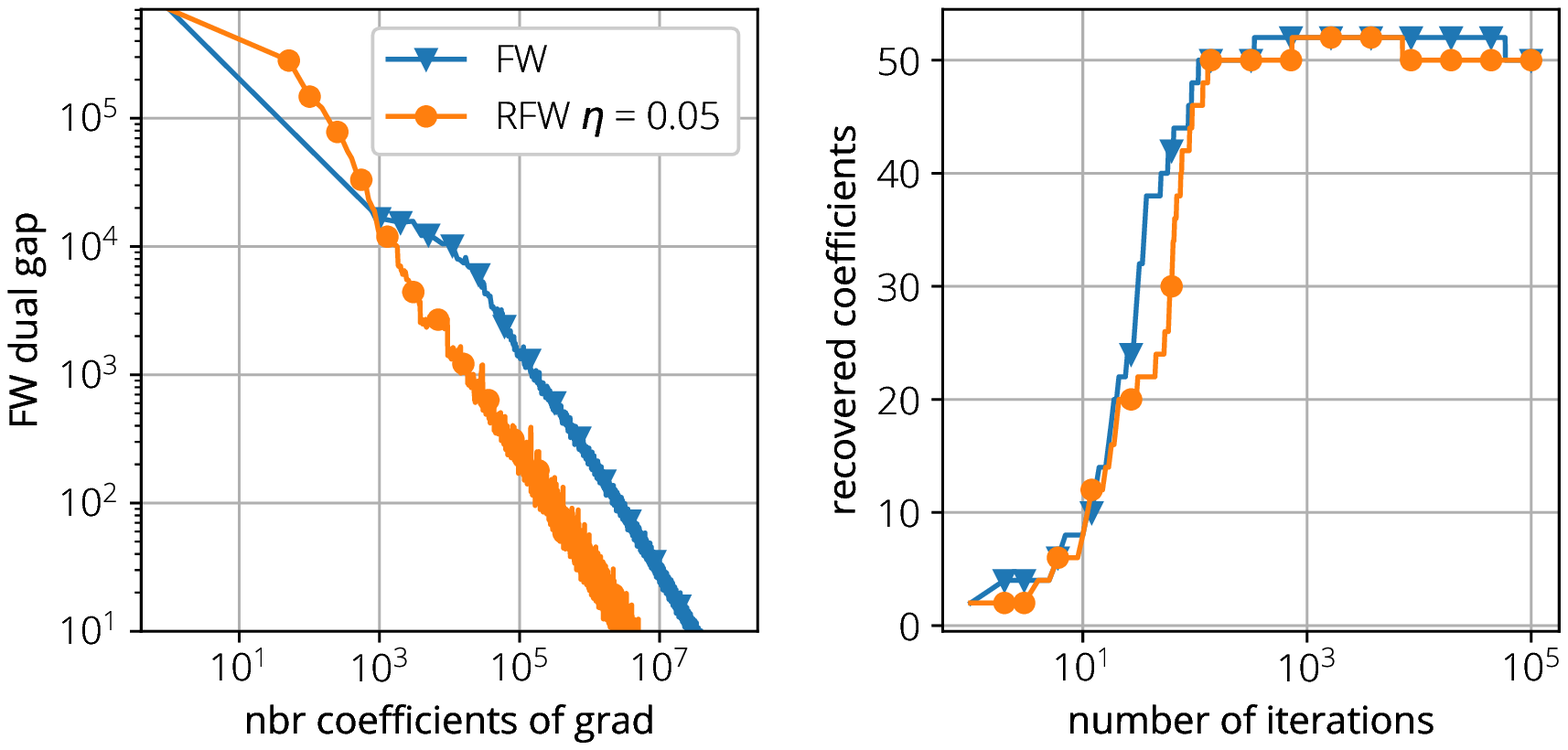}
\end{minipage}

    \caption{Performance of FW against RFW with subsampling parameter $\eta =\frac{p}{|\mathcal{A}|} = 0.05$ (chosen arbitrarily) on the lasso problem. {\em Upper left:} progress in FW dual gap versus number of iterations. {\em Lower left:} progress of the FW dual gap versus cumulative number of computed coefficients of gradient per call to LMO, called \textit{nbr coefficients of grad} here.
    {\em Lower right:} recovered coefficients in support of the ground truth versus number of iterations. {\em Upper right}: size of support of iterate versus number of iterations.} 
    \label{fig:RFW_FW_Lasso}
\end{figure}

\begin{figure}
\begin{minipage}{0.99\linewidth}
\includegraphics[width=\linewidth]{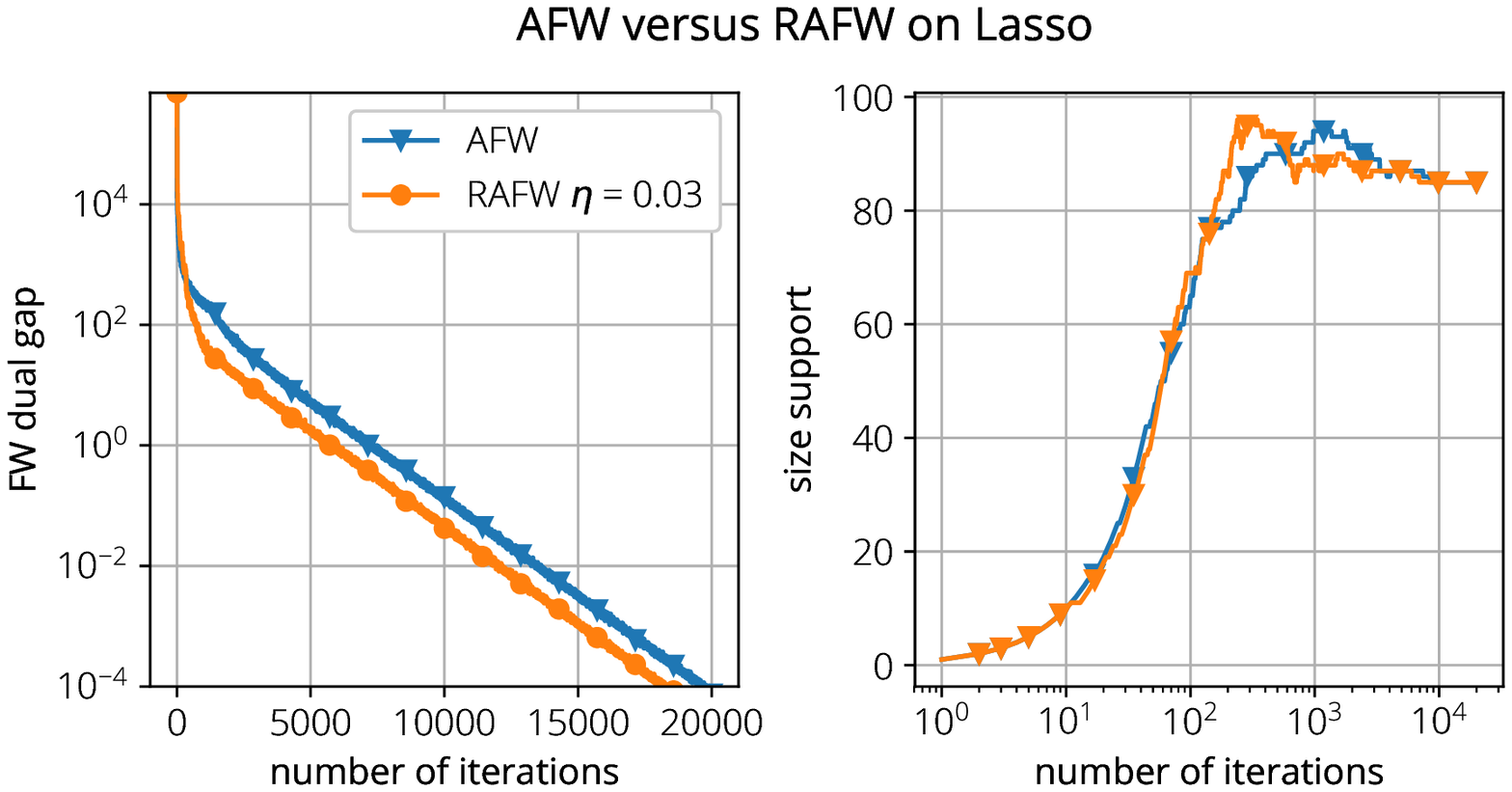}
\end{minipage}

\begin{minipage}{0.99\linewidth}
\includegraphics[width=\linewidth]{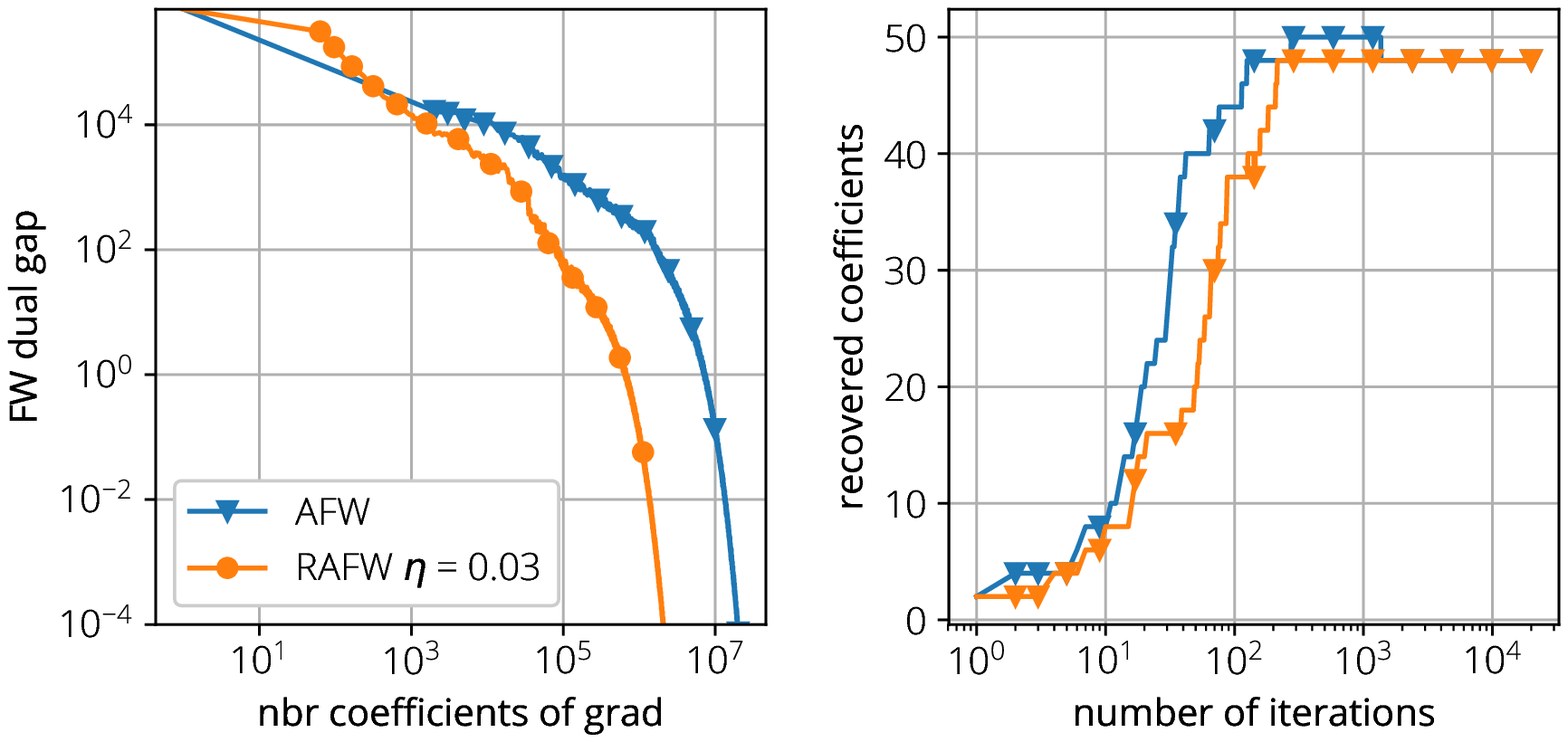}
\end{minipage}

    \caption{Same parameters and setting as in Figure \ref{fig:RFW_FW_Lasso} but to compare RAFW and AFW. AFW performed 880 away steps among which 14 were a drop steps while RAFW performed 1242 away steps and 37 drop steps.
    %add the number of drop step done by the algo?
    % RAFW 1242 away steps 37 drop steps
    % AFW  880 away steps 14 drop steps
    }
    \label{fig:RAFW_AFW_Lasso}
\end{figure}

\paragraph{RAFW experiments.} Unlike RFW, the RAFW method outperforms AFW in terms on number of iterations in the upper left graph in Figure \ref{fig:RAFW_AFW_Lasso}. %This is possible because when the RLMO proposes a not satisfying enough descent direction, there is always the possibility to opt for an away step. 
These graphs also show that both have linear rate of convergence. The bottom left graph shows that the gap between RAFW and AFW is even larger  when comparing the cumulative number of computed coefficients of the gradient required to reach a certain target precision.
%explain why at the beginning it seems to performs not better
% At the beginning of this graph, the RAFW curve is over that of the AFW (same phenomenon on Figure \ref{fig:RFW_FW_Lasso}), but is an artifact of the x-log scale, because the first iteration of AFW requires $d=200$ coefficients.

%explain why it tends to increase number of away steps but not necessarily drop steps.
This out-performance of RAFW over AFW in term of number of iteration to converge is not predicted by our convergence analysis. We conjecture that the away mechanism improves the trade-off between the cost of the LMO and the alignment of the descent direction with the opposite of the gradient. Indeed, because of the oracle subsampling, the partial FW gap (e.g. the scalar product of the Randomized FW direction with the opposite of the gradient) in RAFW is smaller than in the non randomized variant, and so there is a higher likelihood of performing an away step.  
%trade-off at stake in randomizing the LMO (

%Sparsity of the iterate
Finally, the away mechanism enables the support of the RAFW to stay close to that of AFW, which was not the case in the comparison of RFW versus FW. This is illustrated in the right panels of Figure \ref{fig:RAFW_AFW_Lasso}.

\paragraph{Real dataset.}
%\citep[Appendix E]{blondel2016polynomial}
On figure \ref{fig:real_Lasso}, we test the Lasso problem on the E2006-tf-idf data set \citep{kogan2009predicting}, which gathers volatility of stock returns from companies with financial reports. Each financial reports is then represented through its TF-IDF embedding ($n = 16087$ and $d = 8000 $ weafter an initial round of feature selection). The regularizing parameter is chosen to obtain solution with a fraction of $0.01$ nonzero coefficients.

\subsection{Latent Group-Lasso}
% The latent group Lasso 
\paragraph{Notation.}
We write $[d]$ the set of indices from $1$ to $d$. Consider $g\subset [d]$ and $\xx\in\mathbb{R}^d$, $\xx_{(g)}$ represents the projection vector of $\xx$ onto its $g$-coordinate. We use the notation $\nabla_{(g)} f(\xx_t)$ to denote the gradient with respect to the variables in group $g$. Similarly $\xx_{[g]}\in\mathbb{R}^d$ is the vector that equals $\xx$ in the coordinates of $g$ and $0$ elsewhere.

\begin{figure}
\begin{minipage}{0.99\linewidth}
\includegraphics[width=\linewidth]{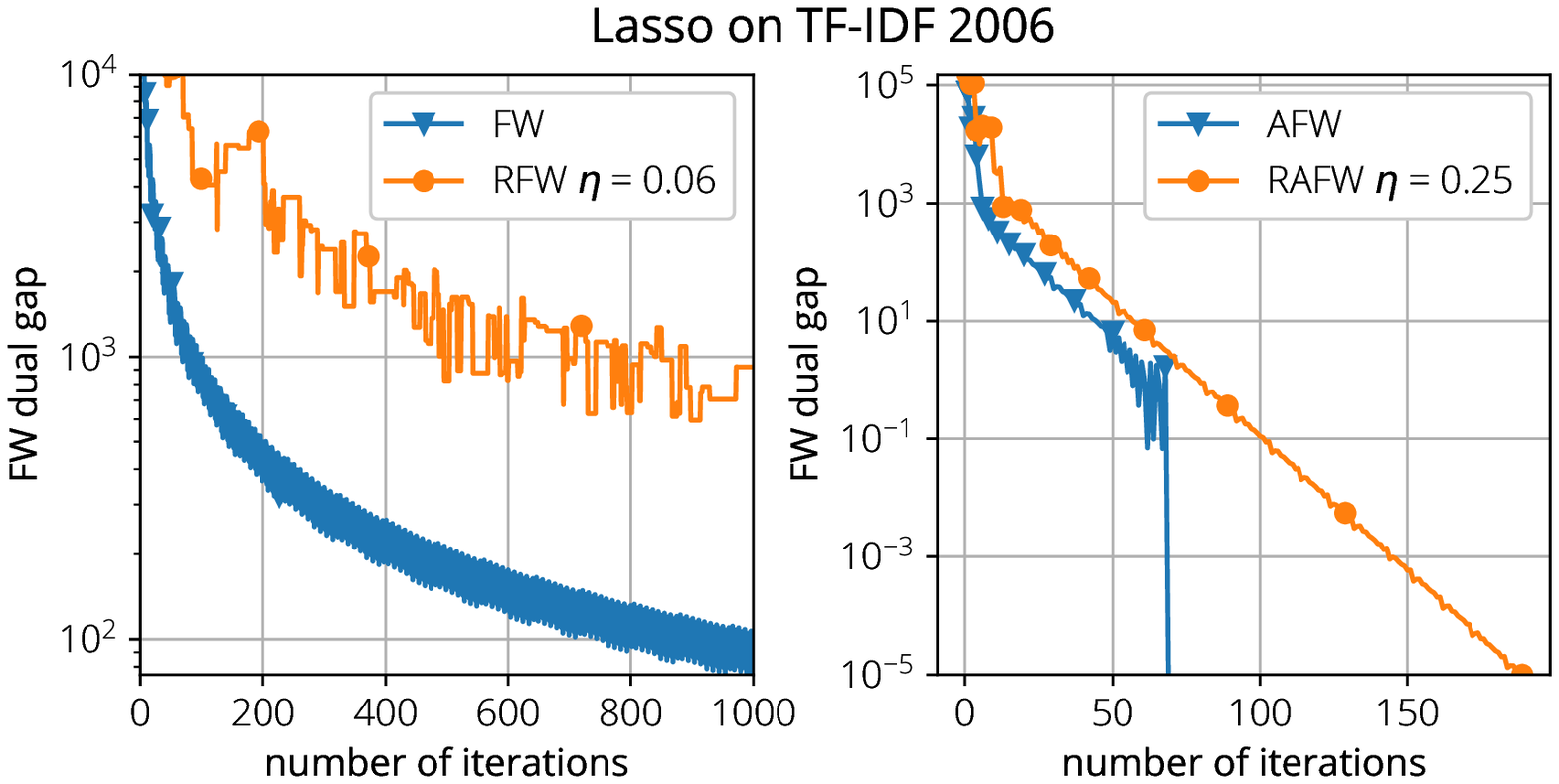}
\end{minipage}

\begin{minipage}{0.99\linewidth}
\includegraphics[width=\linewidth]{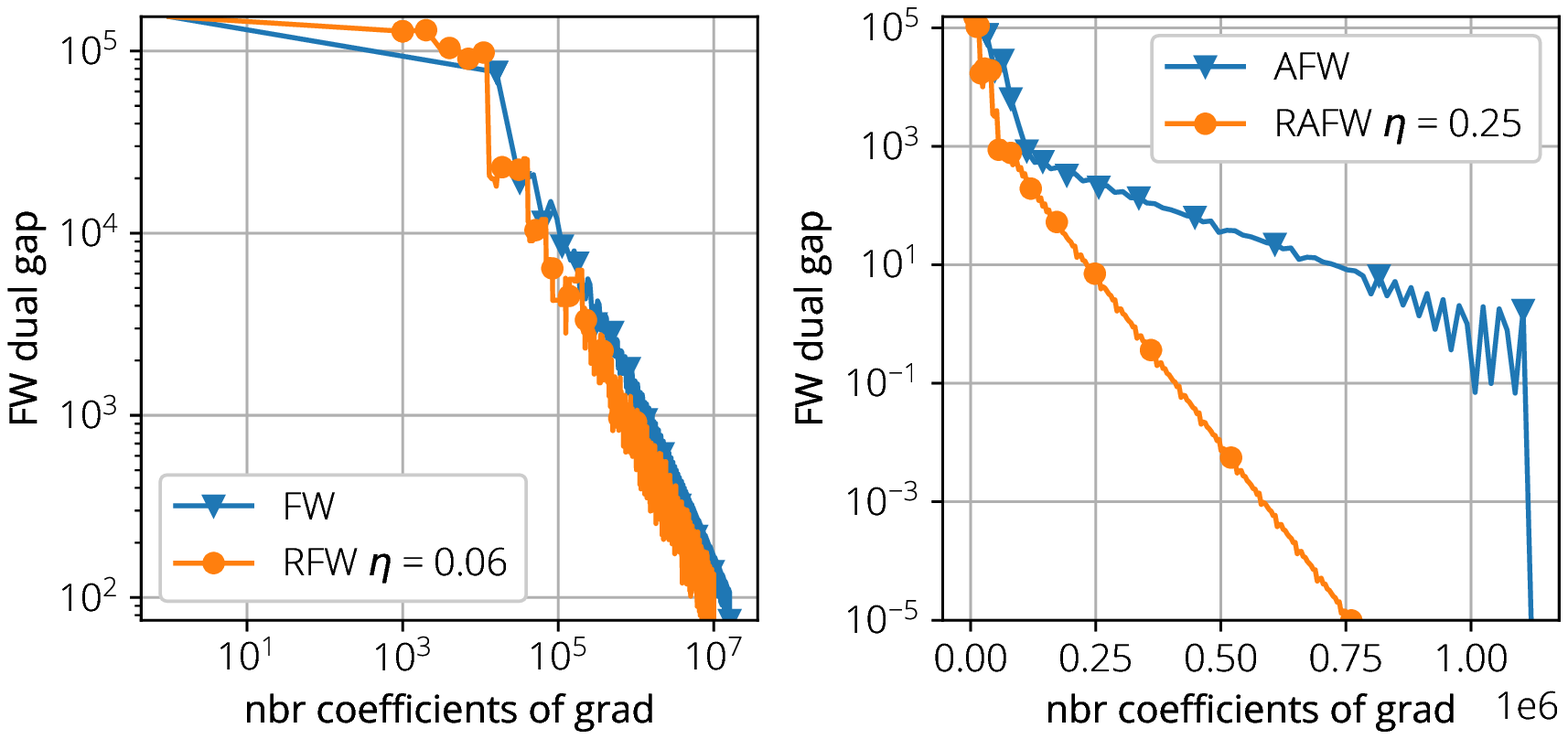}
\end{minipage}

    \caption{Performance of FW and AFW against RFW and RAFW respectively on the lasso problem with TF-IDF 2006 dataset. The subsampling parameter is $\eta =\frac{p}{|\mathcal{A}|} = 0.06$ (again chosen arbitrarily) for RFW and $\eta = 0.25$ for RAFW. {\em Right:} Comparison of RAFW against RFW. {\em Left:} Comparison of RFW against FW.
    {\em Upper:} progress in FW dual gap versus number of iterations. {\em Lower:} progress of the FW dual gap versus cumulative number of computed coefficients in gradient per call to LMO.}
    \label{fig:real_Lasso}
\end{figure}

\paragraph{Model.}
As outlined by \citet{jaggi2013revisiting}, FW algorithms are particularly useful when the domain is a ball of the latent group norm \citep{obozinski2011group}. Consider a set $\mathcal{G}$ of subset of $[d]$ such that $\bigcup_{g\in\mathcal{G}}{g}=[d]$ and denote by $||\cdot||_g$ any norm on $\mathbb{R}^{|g|}$. Frank-Wolfe can be tuned to solve \eqref{eq:obj_fun} with $\mathcal{M}$ being the ball corresponding to the latent group norm
\begin{eqnarray}
||\xx||_{\mathcal{G}}\stackrel{\text{def}}{=} \underset{\scalebox{0.9}{$s.t. ~~ \xx=\sum_{v\in\mathcal{G}}\vv_{[g]}~.$}}{\underset{v_{(g)}\in\mathbb{R}^{|g|}}{\min}     \sum_{g\in\mathcal{G}}||\vv_{(g)}||_g}
\end{eqnarray}
%This formulation was proposed by \cite{jaggi2013revisiting} and 
This formulation matches a constrained version of the regularized \citep[ equation (5)]{obozinski2011group} when each $||\cdot||_g$ is proportional to the Euclidean norm. From now on we will consider $||\cdot||_g$ to be the euclidean norm.

%Link with the group lasso
When  $\mathcal{G}$ forms a partition of $[d]$ (i.e., there is no overlap between groups), this norm coincides with the group lasso norm. %For numerical experiments we will consider $||\cdot||_g$ to be the Euclidean norm for each elements of the group $\mathcal{G}$ with same radius $\beta$.

\paragraph{Sub-sampling.}
%Work on the atomic set to get a subset A_t
Given an element $g$ of $\mathcal{G}$, consider the hyper-disk
\begin{eqnarray*}
\mathbb{D}_g(\beta)=\left\{  \vv\in\mathbb{R}^d~\mid~\vv=\vv_{[g]},||\vv_{(g)}||\leq \beta\right\}~.
\end{eqnarray*}
\citep[lemma 8]{obozinski2011group} shows that such constrain set $\mathcal{M}$ is the convex hull of $\mathcal{A}\stackrel{\text{def}}{=} \underset{g\in\mathcal{G}}{\bigcup}\mathbb{D}_g$.

At iteration $t$ of RFW for a random subset $\mathcal{G}_p$ of size $p$ of $\mathcal{G}$ we then propose to simply run RFW (algorithm \ref{algo:Randomized_FW_general}) with $\mathcal{A}_t\stackrel{\text{def}}{=}\underset{g\in\mathcal{G}_p}{\bigcup}\mathbb{D}_g$. Denoting by $g_p=\underset{g\in\mathcal{G}_p}{\bigcup}g$ the LMO in RFW becomes
\begin{eqnarray*}
\text{LMO}(\xx_t,\mathcal{A}_t)\in\argmax_{\vv \in \mathcal{A}_t}~\langle \vv_{(g_p)}, -\nabla_{(g_p)} f(\xx_t)\rangle~.
\end{eqnarray*}
This means that we only need to compute the gradient on the $g_p$ index. Depending on $\mathcal{G}$ and on the sub-sampling rate, this can be a significant computational benefit. %The euclidean norm being its own dual norm
%\begin{eqnarray*}
%\max_{\vv \in \mathcal{A}_t}~\langle \vv_{(g_p)}, -\nabla_{(g_p)} f(\xx_t)\rangle~=\beta\max_{g\in\mathcal{G}_p}||\nabla_{(g)} f(\xx_t)||~.
%\end{eqnarray*}
%Denote by $g_t\in\argmax_{g\in\mathcal{G}_p}||\nabla_{(g)} f(\xx)||$ and $\nabla f(\xx)=(c_i)_i$, it means that the left hand side is maximized for $\vv$ in $\mathbb{D}_{g_t}$ so that
%\begin{eqnarray*}
%\text{LMO}(\xx_t,\mathcal{A}_t)_i=\frac{-c_i \beta}{ ||\nabla_{(g_t)} f(\xx_t)||}~~\forall i\in g_t~.
%\end{eqnarray*}

\paragraph{Experiments.} We illustrate the convergence speed-up of using RFW over FW for latent group lasso regularized least square regression.

For $d = 10000$ we consider a collection $\mathcal{G}$ of groups of size $10$ with an overlap of $3$ and the associated atomic set $\mathcal{A}$. We chose the ground truth parameter vector $\ww_0\in\text{conv}(\mathcal{A})$ with a fraction of $0.01$ of nonzero coefficients, where on each active group, the coefficients are generated from a Gaussian distribution. The data is a set of $n$ pairs $(y_i,\ww_i)\in\mathbb{R}\times\mathbb{R}^d$ randomly generated from a Gaussian with some additive Gaussian noise. The regularizing parameter is $\beta=14$, set so that the unconstrained optimum lies outside of the constrain set. 

\paragraph{Large dataset and Streaming Model.}
The design matrix is stored in disk. We allow both RFW and FW to access it only through chunks of size $n\times 500$. This streaming model allows a wall clock comparison of the two methods on very large scale problems.%more realistic wall clock comparison of the two methods on very large scale problems.

Computing the gradient when the objective is the least squares loss consists in a matrix vector product. Computing it on a batch of coordinates then requires same operation with a smaller matrix. When computing the gradient at each randomized LMO call, the cost of slicing the design matrix can then compensate the gain in doing a smaller matrix vector product.

With data loaded in memory, which is typically the case for large datasets, both the LMO and the  randomized LMO have this access data cost. Consider also that RFW allows any scheme of sampling, including one that minimizes the cost of data retrieval.

%Indeed if the data were stored in memory, RAFW would have an extra memory-access cost in slicing the design matrix that in practice offsets the computational gain of computing the LMO on a sub-sample set of atoms. This is one reason for the appeal of block coordinate method: when the data is only access via chunk of it, a block coordinate approach do not wait for all the data information to update an iterate.

\begin{figure}[!h]

\includegraphics[width=\linewidth]{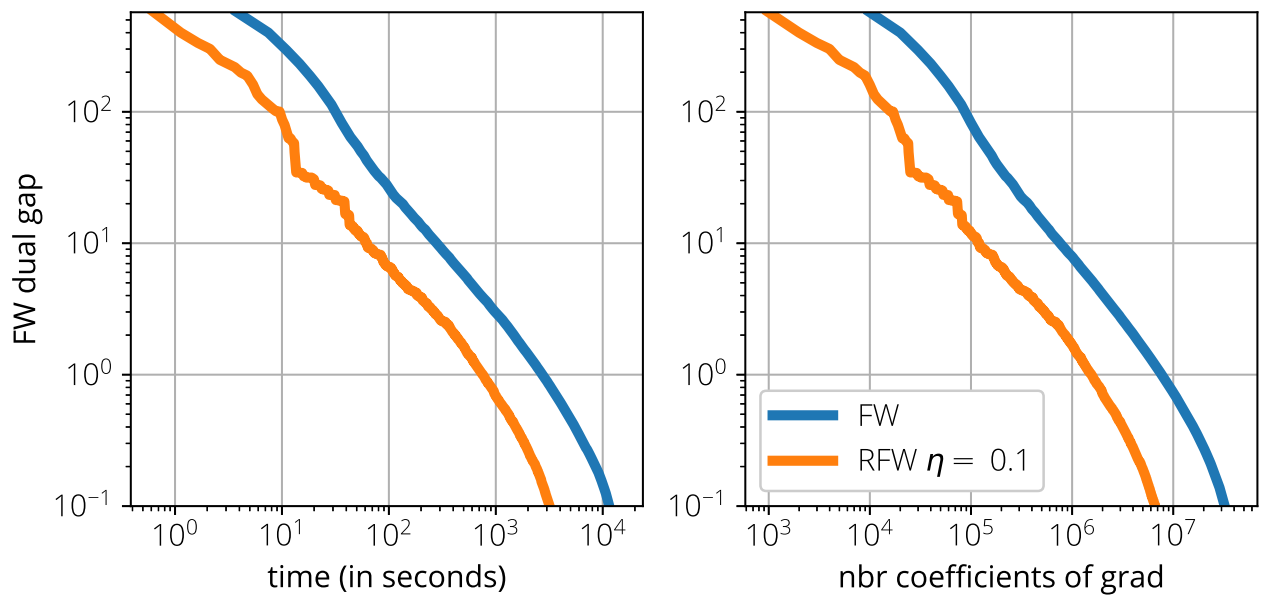}

\caption{Both panels are in log log scale and show convergence speed up for FW and RFW on latent group lasso regularized least square regression. The parameter of subsampling $\eta = 0.1$, is chosen arbitrarily. {\em Left: }evolution of the precision in FW dual gap versus the wall clock time. {\em Right:} evolution of the precision in FW dual gap versus the cumulative number of computed coefficients of the gradient. }\label{fig:RFW_vs_FW_LGL_loss}
\end{figure}

%%%%%%%%%%%%%%%%%%%%%%%%%%%%%%%%%%%%%%%%%%%%%%%%%
%%%%%%%%%%%%%%CONCLUSION%%%%%%%%%%%%%%%%%%%%%%%%%
%%%%%%%%%%%%%%%%%%%%%%%%%%%%%%%%%%%%%%%%%%%%%%%%%

\section{Conclusion and future work}
We give theoretical guarantees of convergence of randomized versions of FW that exhibit same order of convergence as their deterministic counter-parts. As far as we know, for the case of RAFW, this is the first contribution of the kind. While the theoretical complexity bounds don't necessarily imply this, our numerical experiments show that randomized versions often outperform their deterministic ones on $\ell_1$-regularized and latent group lasso regularized least squares. In both cases, randomizing the LMO allows us to compute the gradient only on a subset of its coordinates. We use it to speed up the method in a streaming model where the data is accessed by chunks, but there might be other situations where the structure of the polytope can be leveraged to make subsampling computationally beneficial.

There are also other linearly-convergent FW variants other than AFW, for which it might be possible to derive randomized variants.

Finally many recent results \citep{Goldfarb2016,goldfarb2017linear,hazan2016variance} on FW have combined various improvements of FW (away mechanism, sliding, lazy oracles, stochastic FW, etc.). Randomized oracles add to this toolbox and could further improve its benefits.

\newpage

\ifarxiv
\section*{Acknowledgements}
A.A. is at the d\'epartement d'informatique de l'ENS, \'Ecole normale sup\'erieure, UMR CNRS 8548, PSL Research University, 75005 Paris, France, and INRIA Sierra project-team. T.K. is a PhD student under the supervision of A.A. and acknowledges funding from the CFM-ENS chaire {\em les mod\`eles et sciences des donn\'ees}. 
FP is funded through the European Union's Horizon 2020 research and innovation programme under the Marie Sklodorowska-Curie grant agreement 748900. The authors would like to thanks Robert Gower, Vincent Roulet and Federico Vaggi for helpfull discussions.
\else
\fi

\bibliography{biblio}{}
\bibliographystyle{icml2018}

\clearpage

\appendix
\onecolumn
\gdef\thesection{Appendix \Alph{section}}

\section*{\\\centering\LARGE How to Get Away with Subsampling: a Frank-Wolfe\\
Algorithm for Optimizing over Large Atomic Domains\\
\hfill\\
\centering{Supplementary material}}

\vspace{1em}
%%%%%%%%%%%%%%%%%%%%%%%%%%%%%%%%%%%%%%%%%%%%%%%%%%%%%%%%%%%%%%%%%%%%%%%%%%%%%%%%%%%%%%%%%%%%%%%%%%%%%%%%%%%%%%%%%%%%%%%%%%%%%%%%%%%%%%%%%%%%%%%%%%%
%%%%%%%%%%%%%%%%%%%%%%%%%%%%%%%%%%%%SUB-LINEAR PROOF%%%%%%%%%%%%%%%%%%%%%%%%%%%%%%%%%%%%%%%%%%%%%%%%%%%%%%%%%%%%%%%%%%%%%%%%%%%%%%%%%%%%%%%%%%%
%%%%%%%%%%%%%%%%%%%%%%%%%%%%%%%%%%%%%%%%%%%%%%%%%%%%%%%%%%%%%%%%%%%%%%%%%%%%%%%%%%%%%%%%%%%%%%%%%%%%%%%%%%%%%%%%%%%%%%%%%%%%%%%%%%%%%%%%%%%%%%%%%%%
\paragraph{Appendix notations.}  We denote by $\mathbf{E}_t$ the conditional expectation at iteration $t$, conditioned on all the past and by $\mathbb{E}$ a full expectation. We denote by a tilde the values that come from the deterministic analysis of FW. Denote by $\rr_t=-\nabla f(\xx_t)$.  For $k\in\mathbb{N}^*$, denote by $[k]$ all integer between $1$ and $k$.

\section{Proof of sub-linear convergence for Randomized Frank-Wolfe}\label{apx:rfw_proof}

In this section we provide a convergence proof for Algorithm~\ref{algo:Randomized_FW_general}. The proof is loosely inspired by that of~\citep[Appendix B.1]{locatello17a}, with the obvious difference that the result of the LMO is a random variable in our case.

\begin{thmbis}{th:expectation_rate}
Let $f$ be a function with bounded curvature constant $C_f$,
Algorithm \ref{algo:Randomized_FW_general} for $\eta\in(0,1]$, (with step-size chosen by either variants) converges towards a solution of~\eqref{eq:obj_fun}, satisfying
\begin{equation}
\mathbb{E}(f(\xx_{T}))-f(\xx^*)\leq \frac{2(C_f+f(\xx_0)-f(\xx^*))}{\eta T+2}.
\end{equation}
\end{thmbis}

\begin{proof}
By definition of the curvature constant, at iteration $t$ we have 
\begin{equation}\label{eq:curvature_1}
    f(\xx_{t} + \gamma (\ss_t - \xx_t)) \leq f(\xx_t) + \gamma\langle \nabla f(\xx_t), \ss_t - \xx_t\rangle + \frac{\gamma^2}{2} C_f~.
\end{equation}
By minimizing with respect to $\gamma$ on $[0,1]$ we obtain 
\begin{equation}\gamma_t =\text{clip}_{[0,1]} \langle -\nabla f(\xx_t), \ss_t - \xx_t\rangle / C_f\quad,
\end{equation}
which is the definition of $\gamma_t$ in the algorithm with Variant 2. Hence, we have 
\begin{eqnarray*}
f(\xx_{t+1})&\leq & f(\xx_t) + \min_{\gamma \in [0, 1]}\left\{ \gamma\langle \nabla f(\xx_t), \ss_t - \xx_t\rangle + \frac{\gamma^2}{2} C_f\right\},
\end{eqnarray*}
an inequality which is also valid for Variant 1 since by the line search procedure the objective function at $\xx_{t+1}$ is always equal or smaller than that of Variant 1.
% \begin{eqnarray*}
% f(\xx_{t+1})=\min_{\gamma\in[0,1]} f(\xx_{t} + \gamma (\ss_t - \xx_t))\\
% \leq  f(\xx_t) + \min_{\gamma \in [0, 1]}\left\{ \gamma\langle \nabla f(\xx_t), \ss_t - \xx_t\rangle + \frac{\gamma^2}{2} C_f\right\}
% \end{eqnarray*}
Denote by $h_t = f(\xx_t) - f(\xx^*)$,% for $\xx_{t+1} = \xx_{t} + \gamma_t (\ss_t - \xx_t)$ using the aforementioned step-size we have from Eq.~\eqref{eq:curvature_1}
\begin{equation*}
    h_{t+1} \leq h_t + \min_{\gamma \in [0, 1]}\left\{ \gamma\langle \nabla f(\xx_t), \ss_t - \xx_t\rangle + \frac{\gamma^2}{2} C_f\right\}.
\end{equation*}
We write $\widetilde{\ss}_t$ the FW atom if we had started the FW algorithm at $\xx_t$, and $\Econd_t$ the expectation conditionned on all the past until $\xx_t$, we have
\begin{align}
\Econd_t h_{t+1} &\leq h_t + \Econd_t \min_{\gamma \in [0, 1]}\left\{ \gamma\langle \nabla f(\xx_t), \ss_t - \xx_t\rangle + \frac{\gamma^2}{2} C_f\right\}\\
&\leq h_t + \mathcal{P}(\ss_t = \widetilde{\ss}_t) \min_{\gamma \in [0, 1]}\Big\{\gamma\langle \nabla f(\xx_t), \widetilde{\ss}_t - \xx_t\rangle + \frac{\gamma^2}{2} C_f\Big\}\label{eq:before_FW_dual}\\
&\leq h_t + \eta \min_{\gamma \in [0, 1]}\Big\{-\gamma h(\xx_t) + \frac{\gamma^2}{2} C_f\Big\}~\\
&\leq h_t + \eta \big(-\gamma h(\xx_t) + \frac{\gamma^2}{2} C_f\big)\quad\text{ (for any $\gamma \in [0, 1]$, by definition of $\min$)} \label{eq:for_induction}~,
\end{align}
where the second inequality follows from the definition of expectation and the fact that minimum is non-positive since it is zero for $\gamma=0$. 
%Also to go from \eqref{eq:before_FW_dual} to \eqref{eq:for_induction} uses that the FW gap is an upper bound on the dual gap, e.g. $ \langle -\nabla f(\xx_t), \widetilde{\ss}_t - \xx_t\rangle \geq h(\xx_t)$.
The last inequality is a consequence of uniform sampling as well as it uses that the FW gap is an upper bound on the dual gap, e.g. $ \langle -\nabla f(\xx_t), \widetilde{\ss}_t - \xx_t\rangle \geq h(\xx_t)$.

\textbf{Induction.} From \eqref{eq:for_induction} the following is true for any $\gamma\in[0,1]$
\begin{eqnarray}\label{eq:diff_with_frandi}
\Econd_t(h_{t+1})\leq h_t (1-\eta\gamma)+\frac{\gamma^2}{2}\eta C_f~.
\end{eqnarray}
Taking unconditional expectation and writing $H_t=\mathbb{E}(h_t)$, we get for any $\gamma\in[0,1]$
\begin{eqnarray}\label{eq:final_eq_ref}
H_{t+1}\leq H_t (1-\eta\gamma)+\frac{\gamma^2}{2}\eta C_f.
\end{eqnarray}
With $\gamma_t=\frac{2}{\eta t+2}\in[0,1]$, we get by induction
\begin{eqnarray}\label{eq:result_induction_RFW}
H_t\leq 2\frac{C_f+\epsilon_0}{\eta t+2}=\gamma_t (C_f+\epsilon_0),
\end{eqnarray}
where $\epsilon_0=f(x_0)-f(x^*)$. 
Initialization follows the fact that the curvature constant is positive. For $t>0$, from \eqref{eq:final_eq_ref} and the induction hypothesis
\begin{eqnarray*}
H_{t+1}&\leq& \gamma_t(C_f+\epsilon_0)(1-\eta\gamma_t)+\frac{\gamma_t^2}{2}\eta C_f\\
&\leq& \gamma_t(C_f+\epsilon_0)(1-\eta\gamma_t)+\frac{\gamma_t^2}{2}\eta (C_f+\epsilon_0)\\
&\leq& \gamma_t(C_f+\epsilon_0)(1-\eta\gamma_t+\frac{\gamma_t}{2}\eta)\\
&\leq& (C_f+\epsilon_0)(1-\frac{\gamma_t}{2}\eta)\gamma_t\\
&\leq& (C_f+\epsilon_0)\gamma_{t+1}.
\end{eqnarray*}
The last inequality comes from the fact that $(1-\frac{\gamma_t}{2}\eta)\gamma_t\leq \gamma_{t+1}$. Indeed, with $\gamma_t =\frac{2}{\eta t+2}$, it is equivalent to
\begin{eqnarray*}
(1-\frac{\eta}{\eta t+2})\frac{2}{\eta t+2}&\leq& \frac{2}{\eta (t+1)+2}\\
\Leftrightarrow \frac{(\eta t+2)-\eta}{\eta t+2}&\leq& \frac{\eta t+2}{\eta (t+1)+2}\\
\Leftrightarrow (\eta t+2-\eta)(\eta (t+1)+2)&\leq&  (\eta t+2)^2\\
\Leftrightarrow \eta^2t^2+4\eta t+4-\eta^2&\leq& \eta^2t^2+4\eta t +4.
\end{eqnarray*}
The last being true, it concludes the proof.
\end{proof}

 \clearpage
 \section{Proof of linear convergence for RAFW}\label{apx:linear_proof}

\paragraph{Away curvature and geometric strong convexity}. 
The \emph{away curvature} constant is a modification of the curvature constant described in the previous subsection, in which the FW direction $\ss - \xx$ is replaced with an arbitrary direction $\ss - \vv$:
\begin{equation*}
C_f^A\defas\!\!\!\!\!\!\!\! \underset{\substack{\xx,\ss,\vv\in\mathcal{M}\\\gamma\in [0,1] \\ \yy=\xx+\gamma (\ss-\vv)}}{\text{sup }}\!\!{\frac{2}{\gamma^2}\big( f(\yy)-f(\xx)-\gamma\langle\nabla f(\xx),\ss-\vv\rangle \big)}~.
\end{equation*}
The \emph{geometric strong convexity} constant $\mu_f$ depends on both the function and the domain (in contrast to the standard strong convexity definition)
 and is defined as (see ``An Affine Invariant Notion of Strong Convexity'' in \citep{lacoste2015global} for more details)
\begin{equation*}
  \mu^A_f =\!\!\inf_{\xx \in \mathcal{M}} \!\!\!\!\!\!\underset{\substack{\xx^*\in \mathcal{M} \\ \langle \nabla f(\xx), \xx^* - \xx\rangle < 0}}{\inf}  \frac{2}{\gamma^A(\xx, \xx^*)^2}B_f(\xx, \xx^*)
\end{equation*}
where $B_f(\xx, \xx^*) = f(\xx^*) - f(\xx) - \langle \nabla f(\xx), \xx^* - \xx\rangle$ and $\gamma^A(\xx, \xx^*)$ the positive step-size quantity:
\begin{eqnarray*}
\gamma^A(\xx, \xx^*) \coloneqq \frac{\langle -\nabla f(\xx), \xx^*-\xx \rangle}{\langle -\nabla f(\xx), \ss_f(\xx)-\vv_f(\xx)\rangle}.
\end{eqnarray*}
In particular $\ss_f(\xx)$ is the Frank Wolfe atom starting from $\xx$. $\vv_f(\xx)$ is the away atom when considering all possible expansions of $\xx$ as a convex combinations of atoms in $\mathcal{A}$. Denote by $\mathcal{S}_{\xx} \coloneqq \{\mathcal{S}\mid \mathcal{S}\subseteq\mathcal{A}\text{ such that $\xx$ is a proper convex combination of all elements in $\mathcal{S}$}\}$ and by $\vv_{\mathcal{S}(\xx)} \coloneqq \argmax_{\vv\in\mathcal{S}}\langle\nabla f(\xx),\vv\rangle$. $\vv_f(\xx)$ is finally defined by
\begin{eqnarray*}
\vv_f(\xx)\defas \underset{\{\vv=\vv_{\mathcal{S}}\mid \mathcal{S}\in\mathcal{S}_{\xx}\}}{\argmin} \langle\nabla f(\xx),\vv\rangle~.
\end{eqnarray*}

Similarly following \citep[Lemma 9 in Appendix F]{lacoste2015global}, the geometric $\tilde{\mu}$-generally-strongly-convex constant is defined as
\begin{equation*}
  \tilde\mu_f =\!\!\inf_{\xx \in \mathcal{M}} \!\!\!\!\!\!\underset{\substack{\xx^*\in \scalebox{1}{$\chi$}^* \\ \langle \nabla f(\xx), \xx^* - \xx\rangle < 0}}{\inf}  \frac{1}{2\gamma^A(\xx, \xx^*)^2}\big( f(\xx^*)-f(\xx)-2\langle\nabla f(\xx),\xx^*-\xx\rangle \big)~,
\end{equation*}
where $\scalebox{1}{$\chi$}^*$ represents the solution set of \eqref{eq:obj_fun}.

\paragraph{Notations.} In the context of RAFW, $\mathcal{A}$ denotes the finite set of extremes atoms such that $\mathcal{M}=\text{Conv}(\mathcal{A})$. At iteration $t$, $\mathcal{A}_t$ is a random subset of element of $\mathcal{A}\setminus\mathcal{S}_t$ where $\mathcal{S}_t$ is the current support of the iterate. The Randomized LMO is performed over $\mathcal{V}_t= \mathcal{S}_t\cup\mathcal{A}_t$ so that for Algorithm \ref{algo:Randomized_away_general}, $\ss_t\stackrel{\text{def}}{\in} \argmax_{\vv\in\mathcal{V}_t} \langle - \nabla f(\xx_t), \vv \rangle$ is the FW atom at iteration $t$ for RAFW.

Note that when $|\mathcal{A}\setminus\mathcal{S}_t| \leq p$, Algorithm \ref{algo:Randomized_away_general} does exactly the same as AFW. For the sake of simplicity we will consider that this is not the case. Indeed we would otherwise fall back into the deterministic setting and the proof would just be that of \cite{lacoste2015global}. 

%For readability we will repeat several time the definition of the main variables. From Algorithm \ref{algo:Randomized_away_general}, $\ss_t\stackrel{\text{def}}{=} \argmax_{v\in\mathcal{V}_t} \langle - \nabla f(\xx_t), v - \xx_t\rangle$ is the FW atom at iteration $t$ for RAFW. $\mathcal{V}_t$ is defined in Algorithm \ref{algo:Randomized_away_general} as the union of the current support $\mathcal{S}_t$ of $\xx_t$ with a random subset of size $p$ of atom taken in $\mathcal{A}\setminus\mathcal{S}_t$. Note that when $|\mathcal{V}\setminus\mathcal{S}_t|\leq p$, Algorithm \ref{algo:Randomized_away_general} does exactly the same as AFW. 

We use tilde notation for quantities that are specific to the deterministic FW setting. For instance, $\widetilde\ss_t\stackrel{\text{def}}{\in} \argmax_{\vv\in\mathcal{A}} \langle - \nabla f(\xx_t), \vv \rangle$ is the FW atom for AFW starting at $\xx_t$.

Similarly the Away atom is such that $\vv_t\stackrel{\text{def}}{\in} \argmin_{v\in\mathcal{S}_t} \langle - \nabla f(\xx_t), \vv\rangle$ and it does not depend on the sub-sampling at iteration $t$. Here we do not use any tilde because it is a quantity that appears both in AFW and its Randomized counter-part.

In AFW, $\widetilde{g}_t \defas \langle -\nabla f(\xx_t),\widetilde\ss_t-\vv_t\rangle =\text{max}_{\ss\in\mathcal{A}}\langle -\nabla f(\xx_t),\ss-\vv_t\rangle $ is an upper-bound of the dual gap, named the \textit{pair-wise dual gap} \citep{lacoste2015global}. We consider the corresponding \textit{partial pair-wise dual gap} $\widetilde{g}_t \defas {\langle -\nabla f(\xx_t),\ss_t-\vv_t\rangle} =\underset{\ss\in\mathcal{V}_t}{\text{max }}\langle -\nabla f(\xx_t),\ss-\vv_t\rangle $. It is partial is the sense that the maximum is computed on a subset $\mathcal{V}_t$ of $\mathcal{A}$ which results in the fact that it is not guaranteed anymore to be an upper-bound on the dual-gap.

\paragraph{Structure of the proof.}
The proof is structured around a main part that uses Lemmas \ref{lem:progress_direction_sub_away} and \ref{lemma:proba_conditional_non_drop_step}. Lemma \ref{lemma:proba_worth_case_scenario} is just used to prove Lemma \ref{lemma:proba_conditional_non_drop_step}. 

The main proof follows the scheme of the deterministic one of AFW in \citep[Theorem 8]{lacoste2015global}. It is divided in three parts. The first part consists in upper bounding $h_t\stackrel{\text{def}}{=} f(\xx_t)-f(\xx^*)$ with $\tilde g_t$. It does not depend on the specific construction of the iterates $\xx_t$ and thus remains the same as that in \cite{lacoste2015global}. The second part provides a lower bound on the progress on the algorithm, namely
\begin{eqnarray}
h_{t+1}\leq (1-\rho_f\big(\frac{g_t}{\widetilde{g}_t}\big)^2) h_t,
\end{eqnarray}
with $\rho_f=\frac{\mu_f^A}{4C_f^A}$, when it is not doing a \textit{bad drop step} (defined above). As a proxy for this event, we use the binary variable $z_t$ that equals $0$ for bad drop steps and $1$ otherwise. 

The difficulty lies in that we guarantee a geometrical decrease only when $g_t=\widetilde{g}_t$ and $z_t=1$. Because of the sub-sampling and unlike in the deterministic setting, $z_t$ is a random variable. Lemma \ref{lemma:proba_conditional_non_drop_step} provides a lower bound on the probability of interest, $\mathcal{P}(\widetilde g_t = g_t ~\mid~z_t=1)$, for the last part of the main proof.

Finally, the last part of the proof constructs a bound on the number of times we can expect both $z_t=1$ and $g_t=\widetilde{g}_t$ subject to the constraint that at least half of the iterates satisfy $z_t=1$. It is done by recurrence.

\subsection{Lemmas}
%%%%%%%%%%%%%%%%%%%%%%%%%%%%%%%%%%%%%%%%%%%%%%%%%%%%%%%%%%%%%%%%%
%%%%%%%%%%Lemma to understand how the direction is chosen%%%%%%%
%%%%%%%%%%%%%%%%%%%%%%%%%%%%%%%%%%%%%%%%%%%%%%%%%%%%%%%%%%%%%%%%%
%lemma that gives an understanding of how well the chosen direction is good
This lemma ensures the chosen direction $\dd_t$ in RAFW is a good descent direction, and links it with $g_t$ which may be equal to $\widetilde{g}_t$.
\begin{lemma}\label{lem:progress_direction_sub_away}
Let $\ss_t, \vv_t$ and $\dd_{t}$ be as defined in Algorithm \ref{algo:Randomized_away_general}.
Then for 
$g_{t} \stackrel{\text{def}}{=}\langle - \nabla f(\xx_t), \ss_t - \vv_t\rangle$, we have
\begin{eqnarray}\label{eq:quantify_direction}
\langle - \nabla f(\xx_t),\dd_{t}\rangle \geq \frac{1}{2} g_{t} \geq 0~.
\end{eqnarray}
\end{lemma}

\begin{proof}
The first inequality appeared already in the convergence proof of \citet[Eq. (6)]{lacoste2015global}, which we repeat here for completeness. By the definition of $\dd_{t}$ we have:
\begin{align}
2\langle - \nabla f(\xx_t),\dd_{t}\rangle &\geq  \langle - \nabla f(\xx_t), \dd_{t}^{A}\rangle + \langle - \nabla f(\xx_t), \dd_{t}^{\text{FW}}\rangle\nonumber\\
&=  \langle - \nabla f(\xx_t), \ss_{t}-\vv_{t}\rangle = g_{t}
\end{align}
We only need to prove that $g_{t}$ is non-negative. In line \ref{line:rafw_subsampled_lmo} of algorithm \ref{algo:Randomized_away_general}, $\ss_t$ is the output of LMO performs of the set of atoms  $\mathcal{S}_t\cup \mathcal{A}_t \defas \mathcal{V}_t$,
\begin{eqnarray*}
\ss_{t}=\argmax_{\ss\in\mathcal{V}_t}\langle- \nabla f(\xx_t),\ss\rangle~,
\end{eqnarray*}
so that we have $\langle- \nabla f(\xx_t),\ss_{t}\rangle \geq \langle - \nabla f(\xx_t),\vv_{t}\rangle$. By definition of $g_t$, it implies $g_{t}\geq 0$~.
\end{proof}

%Lemma \ref{lemma:proba_worth_case_scenario} is an auxiliary result used in lemma \ref{lemma:lower_bounding_proba} which is itself an auxiliary result for lemma \ref{lemma:proba_conditional_non_drop_step}.
Lemma \ref{lemma:proba_worth_case_scenario} is just a simple combinatorial result needed in Lemma \ref{lemma:proba_conditional_non_drop_step}. Consider a sequence of $m$ numbers, we lower bound the probability for the maximum of a subset of size greater than $p$ to be equal to the maximum of the sequence.

%B2
%%%%%%%%%%%%%%%%%%%%%%%%%%%%%%%%%%%%%%%%%%%%%%%%%%%%%%%%%%%%%%%%%
%%%%%%%%%%Lemma explaining worth case scenario for proba%%%%%%%%%
%%%%%%%%%%%%%%%%%%%%%%%%%%%%%%%%%%%%%%%%%%%%%%%%%%%%%%%%%%%%%%%%%
\begin{lemma}\label{lemma:proba_worth_case_scenario}
Consider any sequence $(r_i)_{i\in \mathcal{I}}$ in $\mathbb{R}$ with $\mathcal{I}=\{1, \cdots, m\}$, and a subset $\mathcal{I}_p \subseteq \mathcal{I}$ of size $p$. We have %formed by choosing elements from $\mathcal{I}$ uniformly at random. We have
\begin{eqnarray}
\mathcal{P}(\underset{i\in \mathcal{I}_p}{\text{max }} r_i=\underset{i\in \mathcal{I}}{\text{max }} r_i)\geq \frac{p}{m}~.
\end{eqnarray}
\end{lemma}

\begin{proof}
Consider $M=\{i\in \mathcal{I}~|~r_i=\underset{j\in \mathcal{I}}{\text{max }} r_j\}$. We have $\underset{i\in \mathcal{I}_p}{\text{max }} r_i=\underset{i\in \mathcal{I}}{\text{max }} r_i$ if and only if at least one element of $\mathcal{I}_p$ belongs to $M$:
\begin{align}
&\mathcal{P}(\underset{i\in \mathcal{I}_p}{\text{max }} r_i=\underset{i\in \mathcal{I}}{\text{max }} r_i)=\mathcal{P}(|\mathcal{I}_p\cap M|\geq 1)~.
\end{align}
By definition $M$ has at least one element $i_0$. Since $\{i_0 \in \mathcal{I}_p\}\subset \{|\mathcal{I}_p\cap M|\geq 1\}$
\begin{align}
&\mathcal{P}(|\mathcal{I}_p\cap M|\geq 1)\geq \mathcal{P}(\{i_0 \in \mathcal{I}_p\})~.
\end{align}
All subsets are taken uniformly at random, we just have to count the number of subset $\mathcal{I}_p$ of $\mathcal{I}$ of size $p$ with $i_0 \in \mathcal{I}_p$%. It requires the choice of one index $i_M$ in $M$ and then to sample a subset of size $p-1$ in $I\setminus \{i_M\}$
\begin{align}
&\mathcal{P}(\{i_0 \in \mathcal{I}_p\})=\frac{{{m-1}\choose{p-1}}}{{{m}\choose{p}}}= \frac{p}{m}\\
&\mathcal{P}(\underset{i\in \mathcal{I}_p}{\text{max }} r_i=\underset{i\in \mathcal{I}}{\text{max }} r_i)\geq \frac{p}{m}~.
\end{align}

\end{proof}

%%%%%%%%%%%%%%%%%%INTRODUCTION%%%%%%%%%%%%%%%%%%%%%%%%%%%%%%%%%%%%
In the second part of the main proof we ensure a geometric decrease when both $g_t=\widetilde{g}_t$ and $z_t=1$, i.e. outside of \textit{bad drop steps}. The following lemma helps quantifying the probability of $g_t=\widetilde{g}_t$ holding when $z_t=1$.

%%%%%%%%%%%%%%%%%%%%%%%%%%%%%%%%%%%%%%%%%%%%%%%%%%%%%%%%%%%%%%%%%
%%%%%%%%%%Lemma g_t=tilde g_t%%%%%%%%%%%%%%%%%%%%%%%%%%%%%%%%%%%%
%%%%%%%%%%%%%%%%%%%%%%%%%%%%%%%%%%%%%%%%%%%%%%%%%%%%%%%%%%%%%%%%%
\begin{lemma}\label{lemma:proba_conditional_non_drop_step}
Consider $g_t$ (defined in Lemma~\ref{lem:progress_direction_sub_away}) to be the partial pair-wise (PW) dual gap of RAFW at iteration $t$ with sub-sampling parameter $p$ on the constrained polytope $\mathcal{M}=\conv(\mathcal{A})$, where $\mathcal{A}$ is a finite set of extremes points of $\mathcal{M}$. $\widetilde{g}_t \defas \underset{\ss\in\mathcal{A}}{\text{max }}\langle -\nabla f(\xx_t),s-\vv_t\rangle$ is the pairwise dual gap of AFW starting at $\xx_t$ on this same polytope. Denote by $z_t$ the binary random variable that equals $0$ when the $t^{th}$ iteration of RAFW makes an away step that is a drop step with $\gamma_{max}<1$ (a bad drop step), and $1$ otherwise. Then we have the following bound %on the probability of not making a bad drop step
%\begin{eqnarray}
%\mathcal{P}(g_t=\widetilde{g}_t~|~\xx_t,z_t=1)\geq \Big(\frac{p}{|\mathcal{A}|}\Big)^2.
%\end{eqnarray}
\begin{empheq}[box=\mybluebox]{equation}\tag{PROB}\label{eq:proba}
\mathcal{P}(g_t=\widetilde{g}_t~|~\xx_t,z_t=1)\geq \Big(\frac{p}{|\mathcal{A}|}\Big)^2.
\end{empheq}

\end{lemma}

\begin{proof}
%INTRODUCTION OF THE PARTITION ON A_I
Recall that $g_t^A \defas \langle \rr_t,\dd_t^A\rangle$. By definition $\{z_t=0\}=\{g_t<g_t^A,\gamma_{\text{max}}<1,\gamma_t^*=\gamma_{\text{max}}\}$, where $\gamma_t^* \defas \argmin_{\gamma\in[0,\gamma_{\text{max}}]} f(\xx_t + \gamma \dd_t^A)$. Its complementary $\{z_t=1\}$ can thus be expressed as the partition $A_1 \cup A_2 \cup A_3$ where the $A_i$ are defined by
\begin{eqnarray}
A_1&=&\{g_t\geq g_t^A\}\quad \text{ (performs a FW step)}\\
A_2&=&\{g_t < g_t^A~,~\alpha_{\vv_t}^{(t)}/(1-\alpha_{\vv_t}^{(t)})\geq 1\} \quad \text{ (performs away step with $\gamma_{\text{max}} \geq 1$)}\\
A_3&=&\{g_t < g_t^A~,~\alpha_{\vv_t}^{(t)}/(1-\alpha_{\vv_t}^{(t)})<1~,~\gamma_t^*<\alpha_{\vv_t}^{(t)}/(1-\alpha_{\vv_t}^{(t)})\}.
\end{eqnarray}

First note that in the case of $A_2$ and $A_3$, $\gamma_{\text{max}}=\alpha_{\vv_t}^{(t)}/(1-\alpha_{\vv_t}^{(t)})$. Though the right hand side formulation highlights that it is entirely determined by $\xx_t$, recalling that $\alpha_{\vv_t}^{(t)}$ is the mass along the atom $\vv_t$ in the decomposition of $\xx_t$ in \S\ref{s:rafw}.

 From a higher level perspective, these cases are those for which we can guarantee a geometrical decrease of $h_t= f(\xx_t)-f(\xx^*)$ (see second part of main proof). By definition, the $A_i$ are disjoints. $A_1$ represents a choice of a FW step in RAFW contrary to $A_2$ and $A_3$ which stands for an away step choice in RAFW. $A_2$ is an away step for which there is enough potential mass ($\gamma_{\text{max}}>1$) to move along the away direction and to ensure sufficient objective decreasing. $A_3$ encompasses the situations where there is not a lot of mass along the away direction ($\gamma_{\text{max}}<1$) but which is not a drop step, e.g. the amount of mass is not a limit to the descent.

%EXPLANATION OF GOAL
Our goal is to lower bound $P=\mathcal{P}(g_t=\widetilde{g}_t~|~\xx_t,z_t=1)$. The following probabilities will be with respect to the $t^{th}$ sub-sampling only. Notice that $g_t^A$, $\widetilde{g}_t$ and $\alpha_{v_t}$ are known given $\{\xx_t,z_t=1\}$. Using Bayes' rule, and because the $A_i$ are disjoints, we have
\begin{eqnarray}
P&=&\mathcal{P}(g_t=\widetilde{g}_t~|~\xx_t,\{z_t=1\})\nonumber\\
&=&\frac{\sum_{i=1}^3{\mathcal{P}(g_t=\widetilde{g}_t~|~\xx_t,A_i)\mathcal{P}(A_i~|~\xx_t)}}{\sum_{i=1}^3{\mathcal{P}(A_i~|~\xx_t)}}~.\label{eq:proba_bayes_decomposition}
\end{eqnarray}

%WHY A CASE DISTINCTION
By definition of $g_t$ and $\widetilde{g}_t$, $g_t\leq \widetilde{g}_t$, so that measuring the probability of an event like $\{g_t=\widetilde{g}_t\}$ conditionally on $\{g_t\leq g_t^A\}$ will naturally depend on whether or not, the deterministic condition $\widetilde{g}_t\geq g_t^A$ is satisfied. Hence the following case distinction.

Recall $\mathcal{V}_t =\mathcal{S}_t\cup\mathcal{A}_t$.

%%%%%%%%%%%%%%%FIRST CASE%%%%%%%%%%%%%%%%%%%%%%%%%%%%%%%%%%%%%%%%%%%%%%%%%
\noindent\textbf{Case $\widetilde{g}_t < g_t^A$.}
\begin{eqnarray}
P=\frac{\sum_{i=1}^3{\mathcal{P}(g_t=\widetilde{g}_t~|~x_t,A_i,\widetilde{g}_t < g_t^A)\mathcal{P}(A_i~|~x_t,\widetilde{g}_t < g_t^A)}}{\sum_{i=1}^3{\mathcal{P}(A_i~|~x_t,\widetilde{g}_t < g_t^A)}}~.
\end{eqnarray}
Recall that $A_1=\{g_t\geq g_t^A\}$. Since by definition $g_t\leq \widetilde{g}_t$, conditionally on $\{\widetilde{g}_t < g_t^A\}$, the probability of $A_1$ is zero. Consequently the above reduces to
\begin{eqnarray}
P&=&\frac{\sum_{i=2}^3{\mathcal{P}(g_t=\widetilde{g}_t~|~x_t,A_i,\widetilde{g}_t < g_t^A)\mathcal{P}(A_i~|~x_t,\widetilde{g}_t < g_t^A)}}{\sum_{i=2}^3{\mathcal{P}(A_i~|~x_t,\widetilde{g}_t < g_t^A)}}\nonumber\\
&\geq& \frac{p}{|\mathcal{A}|}     \frac{\sum_{i=2}^3{\mathcal{P}(A_i~|~x_t,\widetilde{g}_t < g_t^A)}}
{\sum_{i=2}^3{\mathcal{P}(A_i~|~x_t,\widetilde{g}_t \leq g_t^A)}}
=\frac{p}{|\mathcal{A}|}~\label{eq:first_case}.
\end{eqnarray}

%% INTERLUDE
Where the last inequality is because for $i=2,3$ we have $\mathcal{P}(g_t=\widetilde{g}_t~|~x_t,A_i,\widetilde{g}_t < g_t^A)\geq \frac{p}{|\mathcal{A}|}$. Indeed for $A_3$ (case $A_2$ is similar) denote
\begin{eqnarray}
P_1&=&\mathcal{P}(g_t=\widetilde{g}_t~|~x_t,A_3,\widetilde{g}_t < g_t^A)\label{eq:lower_bound_proba_A_3}\\
&=&\mathcal{P}(\underset{\ss\in\mathcal{V}_t}{\text{max }}\langle \rr_t,\ss\rangle=\underset{\ss\in\mathcal{A}}{\text{max }}\langle \rr_t,\ss\rangle~|~\xx_t, \underset{\ss\in\mathcal{V}_t}{\text{max }}\langle \rr_t,\ss\rangle < C_0 ,\underset{\ss\in\mathcal{A}}{\text{max }}\langle \rr_t,\ss\rangle<C_0,\alpha_{\vv_t}^{(t)}/(1-\alpha_{\vv_t}^{(t)})< 1, \gamma_t^*<\alpha_{\vv_t}^{(t)}/(1-\alpha_{\vv_t}^{(t)}))~.\nonumber
\end{eqnarray}
with  $C_0 \defas g_t^A+\langle \rr_t,\vv_t\rangle$ and $\rr_t= -\nabla f(\xx_t)$ not depending on the $t^{th}$ sub-sampling. Also the event $\{\underset{\ss\in\mathcal{V}_t}{\text{max }}\langle \rr_t,\ss\rangle < C_0\}$ is a consequence of $\{\underset{\ss\in\mathcal{A}}{\text{max }}\langle \rr_t,\ss\rangle < C_0\}$ so that $P_1$ simplifies to
\begin{eqnarray}
P_1&=&\mathcal{P}(\underset{\ss\in\mathcal{V}_t}{\text{max }}\langle \rr_t,\ss\rangle=\underset{\ss\in\mathcal{A}}{\text{max }}\langle \rr_t,\ss\rangle~|~\xx_t ,\underset{\ss\in\mathcal{A}}{\text{max }}\langle \rr_t,\ss\rangle < C_0,\alpha_{\vv_t}^{(t)}/(1-\alpha_{\vv_t}^{(t)})< 1, \gamma_t^*<\alpha_{\vv_t}^{(t)}/(1-\alpha_{\vv_t}^{(t)}))\label{eq:last_test}~.
%&=&\mathcal{P}(\underset{\ss\in\mathcal{V}_t}{\text{max }}\langle \rr_t,\ss\rangle=\underset{\ss\in\mathcal{A}}{\text{max }}\langle \rr_t,\ss\rangle~|~(\langle \rr_t,\ss\rangle)_{\ss\in\mathcal{A}}),\label{eq:last_test}
\end{eqnarray}
 By definition
\begin{eqnarray}
\gamma_t^*\in\underset{\gamma\in[0,\frac{\alpha_{\vv_t}^{(t)}}{1-\alpha_{\vv_t}^{(t)}}]}{\argmin~} f(\xx_t+\gamma \dd_t^A)~,
\end{eqnarray}
so that $\gamma_t^*$ does not depend on the $t^{th}$ sub-sampling. Finally all the conditioning in the probability of \eqref{eq:last_test} do not depend on this $t^{th}$ sub-sampling. Hence we are in the position of using Lemma~\ref{lemma:proba_worth_case_scenario} for the sequence $(\langle \rr_t, \ss\rangle)_{\ss\in\mathcal{A}}$. Also by definition of $\mathcal{V}_t=\mathcal{S}_t\cup\mathcal{A}_t$, we have $|\mathcal{V}_t|\geq p$ so that we finally get
\begin{eqnarray}
\mathcal{P}(g_t=\widetilde{g}_t~|~x_t,A_3,\widetilde{g}_t < g_t^A) \geq \frac{p}{|\mathcal{A}|}~.
\end{eqnarray}
This was what was needed to conclude \eqref{eq:first_case}.

%We have by definition in RAFW, $|\mathcal{V}_t|\geq p$. Because Lemma \ref{lemma:proba_worth_case_scenario} applies on any sequence of $|\mathcal{A}|$ numbers, it applies in particular to $(\langle \rr_t, \ss\rangle)_{\ss\in\mathcal{A}}$ satisfying the constraints expressed in the conditioning of the probability in \eqref{eq:last_test}. We thus finally get 
%\begin{eqnarray}
%\mathcal{P}(g_t=\widetilde{g}_t~|~x_t,A_3,\widetilde{g}_t < g_t^A) \geq \frac{p}{|\mathcal{A}|}~.
%\end{eqnarray}
%This was what was needed to conclude \eqref{eq:first_case}.

%%%%%%%%%%%%%%%SECOND CASE%%%%%%%%%%%%%%%%%%%%%%%%%%%%%%%%%%%%%%%%%%%%%%%%%%
\noindent\textbf{Case $\widetilde{g}_t \geq g_t^A$.} In such a case, $P$ from \eqref{eq:proba_bayes_decomposition} rewrites as
\begin{eqnarray}
P=\frac{\sum_{i=1}^3{\mathcal{P}(g_t=\widetilde{g}_t~|~x_t,A_i,\widetilde{g}_t \geq g_t^A)\mathcal{P}(A_i~|~x_t,\widetilde{g}_t \geq g_t^A)}}{\sum_{i=1}^3{\mathcal{P}(A_i~|~x_t,\widetilde{g}_t \geq g_t^A)}}~.
\end{eqnarray}
Here $\mathcal{P}(g_t=\widetilde{g}_t~|~x_t,A_i,\widetilde{g}_t \geq g_t^A)=0$ for $i=2,3$ because $A_i$ implies $g_t < g_t^A$. So that when $\widetilde{g}_t \geq g_t^A$ it is then impossible for $g_t$ to equal $\widetilde{g}_t$. 
\begin{eqnarray*}
P=\frac{\mathcal{P}(g_t=\widetilde{g}_t~|~x_t,A_1,\widetilde{g}_t \geq g_t^A)\mathcal{P}(A_1~|~x_t,\widetilde{g}_t \geq g_t^A)}{\sum_{i=1}^3{\mathcal{P}(A_i~|~x_t,\widetilde{g}_t \geq g_t^A)}}~.
\end{eqnarray*}

Here also we use, and prove later on (see \S below the conclusion of the proof of the Lemma), the lower bound
\begin{eqnarray}\label{eq:lower_bound_proba_g}
\mathcal{P}(g_t=\widetilde{g}_t~|~x_t,A_1,\widetilde{g}_t \geq g_t^A)\geq \frac{p}{|\mathcal{A}|}~,
\end{eqnarray}
that implies
\begin{eqnarray*}
P\geq \frac{p}{|\mathcal{A}|}\frac{\mathcal{P}(A_1~|~x_t,\widetilde{g}_t \geq g_t^A)}{\sum_{i=1}^3{\mathcal{P}(A_i~|~x_t,\widetilde{g}_t \geq g_t^A)}}~.
\end{eqnarray*}
Because the $A_i$ are disjoint, $\sum_{i=1}^3{\mathcal{P}(A_i~|~x_t,\widetilde{g}_t \geq g_t^A)}\leq 1$ we have
\begin{eqnarray*}
P\geq\frac{p}{|\mathcal{A}|}\mathcal{P}(A_1~|~x_t,\widetilde{g}_t \geq g_t^A)~.
\end{eqnarray*}
Using a similar lower bound as \eqref{eq:lower_bound_proba_g}, namely 
\begin{eqnarray}\label{eq:lower_bound_proba_A_1}
\mathcal{P}(A_1~|~x_t,\widetilde{g}_t \geq g_t^A)\geq \frac{p}{|\mathcal{A}|}~,
\end{eqnarray}
we finally get
\begin{eqnarray}\label{eq:lower_bounding_last_case}
P\geq \Big(\frac{p}{|\mathcal{A}|}\Big)^2~.
\end{eqnarray}

%%%%%%%%%%%%%%%CONCLUSION%%%%%%%%%%%%%%%%%%%%%%%%%%%%%%%%%%%%%%%%%%%%%%%%%%
Since it is hard to precisely count the occurrences of $\{\widetilde{g}_t\geq g_t^A\}$ and $\{\widetilde{g}_t<g_t^A\}$, we use a conservative bound in \eqref{eq:lower_bounding_last_case}
%\begin{eqnarray}
%\mathcal{P}(g_t=\tilde g_t~|~\xx_t,z_t=1)\geq \big(\frac{p}{|\mathcal{A}|}\big)^2~.
%\end{eqnarray}
\begin{empheq}[box=\mybluebox]{equation}
\mathcal{P}(g_t=\widetilde{g}_t~|~\xx_t,z_t=1)\geq \Big(\frac{p}{|\mathcal{A}|}\Big)^2~.
\end{empheq}
This will of course make our bound on the rate of convergence very conservative.

%%%%%%%%%%%%%%%SOME JUSTIFICATION%%%%%%%%%%%%%%%%%%%%%%%%%%%%%%%%%%%%%%%%%%%%%%%%%%
\noindent\textbf{Justification for \eqref{eq:lower_bound_proba_g} and \eqref{eq:lower_bound_proba_A_1}}.

%%%%%%%%%%%%%%%%%%%%
%%%%%%%% JUSTIFICATION 1
%%%%%%%%%%%%%%%%%%%%
Lets denote the left hand side of\eqref{eq:lower_bound_proba_g} by $P_2$. By definition of $g_t$ and $\widetilde{g}_t$, with $\rr_t=-\nabla f(\xx_t)$, we have:
\begin{align}
P_2&=\mathcal{P}(\underset{\ss\in\mathcal{V}_t}{\text{max }}\langle \rr_t,\ss-\vv_t\rangle=\underset{\ss\in\mathcal{A}}{\text{max }}\langle \rr_t,\ss-\vv_t\rangle~|~\xx_t, \underset{\ss\in\mathcal{V}_t}{\text{max }}\langle \rr_t,\ss-\vv_t\rangle \geq g_t^A,\underset{\ss\in\mathcal{A}}{\text{max }}\langle \rr_t,\ss-\vv_t\rangle\geq g_t^A)\\
&=\mathcal{P}(\underset{\ss\in\mathcal{V}_t}{\text{max }}\langle \rr_t,\ss\rangle=\underset{\ss\in\mathcal{A}}{\text{max }}\langle \rr_t,\ss\rangle~|~\xx_t, \underset{\ss\in\mathcal{V}_t}{\text{max }}\langle \rr_t,\ss\rangle \geq C_0 ,\underset{\ss\in\mathcal{A}}{\text{max }}\langle \rr_t,\ss\rangle\geq C_0)~,
\end{align}
where $C_0 \defas g_t^A+\langle \rr_t,v_t\rangle$ and $\rr_t$ does not depend on the random sampling at iteration $t$. Bayes formula leads to
\begin{align}
P_2=\frac{\mathcal{P}(\{\underset{\ss\in\mathcal{V}_t}{\text{max }}\langle \rr_t,\ss\rangle=\underset{\ss\in\mathcal{A}}{\text{max }}\langle \rr_t,\ss\rangle \}\cap\{ \underset{\ss\in\mathcal{V}_t}{\text{max }}\langle \rr_t,\ss\rangle \geq C_0\} ~|~\xx_t,\underset{\ss\in\mathcal{A}}{\text{max }}\langle \rr_t,\ss\rangle\geq C_0)}{\mathcal{P}(\underset{\ss\in\mathcal{V}_t}{\text{max }}\langle \rr_t,\ss\rangle \geq C_0 ~|~\xx_t,\underset{\ss\in\mathcal{A}}{\text{max }}\langle \rr_t,\ss\rangle\geq C_0)}~.
\end{align}
Conditionally on $\{\underset{\ss\in\mathcal{A}}{\text{max }}\langle \rr_t,\ss\rangle\geq C_0\}$, the event
$\{\underset{\ss\in\mathcal{V}_t}{\text{max }}\langle \rr_t,\ss\rangle=\underset{\ss\in\mathcal{A}}{\text{max }}\langle \rr_t,\ss\rangle \}$ implies$ \{\underset{\ss\in\mathcal{V}_t}{\text{max }}\langle \rr_t,\ss\rangle \geq C_0\}$ which leads to
\begin{align*}
P_2&=\frac{\mathcal{P}(\underset{\ss\in\mathcal{V}_t}{\text{max }}\langle \rr_t,\ss\rangle=\underset{\ss\in\mathcal{A}}{\text{max }}\langle \rr_t,\ss\rangle ~|~\xx_t,\underset{\ss\in\mathcal{A}}{\text{max }}\langle \rr_t,\ss\rangle\geq C_0)}{\mathcal{P}(\underset{\ss\in\mathcal{V}_t}{\text{max }}\langle \rr_t,\ss\rangle \geq C_0 ~|~\xx_t,\underset{\ss\in\mathcal{A}}{\text{max }}\langle \rr_t,\ss\rangle\geq C_0)}\\
&\geq  \mathcal{P}(\underset{\ss\in\mathcal{V}_t}{\text{max }}\langle \rr_t,\ss\rangle=\underset{\ss\in\mathcal{A}}{\text{max }}\langle \rr_t,\ss\rangle ~|~\xx_t,\underset{\ss\in\mathcal{A}}{\text{max }}\langle \rr_t,\ss\rangle\geq C_0)\geq \frac{p}{|\mathcal{A}|}~,
%&\geq\mathcal{P}(\underset{\ss\in\mathcal{V}_t}{\text{max }}\langle \rr_t,\ss\rangle=\underset{\ss\in\mathcal{A}}{\text{max }}\langle \rr_t,\ss\rangle~|~(\langle \rr_t,\ss\rangle)_{\ss\in\mathcal{A}})\geq \frac{p}{|\mathcal{A}|}~,
\end{align*}
where the last inequality is a consequence of applying Lemma $2$ on the sequence $(\langle\rr_t,\ss\rangle)_{\ss\in\mathcal{A}}$

%%%%%%%%%%%%%%%%%%%%
%%%%%%%% JUSTIFICATION 2
%%%%%%%%%%%%%%%%%%%%
Similarly let's denote the left hand side of \eqref{eq:lower_bound_proba_A_1} by $P_3$. The first inequality is justified because conditionally on $\{\widetilde{g}_t\geq g_t^A\}$, $\{g_t=\widetilde{g}_t\}\subset\{g_t\geq g_t^{A}\}$ and the last by applying, similarly as for \eqref{eq:lower_bound_proba_g}, Lemma \ref{lemma:proba_worth_case_scenario} on the sequence $(\langle \rr_t,\ss\rangle)_{\ss\in\mathcal{A}}$.
\begin{eqnarray*}
P_3&=&\mathcal{P}(g_t\geq g^A_t~|~\xx_t , \widetilde{g}_t \geq g_t^{A})\\
&\geq& \mathcal{P}(g_t=\widetilde{g}_t~|~\xx_t,\widetilde{g}_t\geq g_t^A),\\
&\geq&\mathcal{P}(\underset{\ss\in\mathcal{V}_t}{\text{max }}\langle \rr_t,\ss\rangle=\underset{\ss\in\mathcal{A}}{\text{max }}\langle \rr_t,\ss\rangle~|~\xx_t, \underset{\ss\in\mathcal{A}}{\text{max }}\langle \rr_t,\ss\rangle\geq C_0)\\
%&\geq&\mathcal{P}(\underset{\ss\in\mathcal{V}_t}{\text{max }}\langle \rr_t,\ss\rangle=\underset{\ss\in\mathcal{A}}{\text{max }}\langle \rr_t,\ss\rangle~|~(\langle \rr_t,\ss\rangle)_{\ss\in\mathcal{A}})\\
&\geq&\frac{p}{|\mathcal{A}|}~.
\end{eqnarray*}

\end{proof}

\subsection{Main proof}
%%%%%%%%%%%%%%%%%%%%%%%%%%%%%%%%%%%%%%%%%%%%%%%%%%%%%%%%%%%%%%%%%
%%%%%%%%%%MAIN THEOREM AND PROOF%%%%%%%%%%%%%%%%%%%%%%%%%%%%%%%%%
%%%%%%%%%%%%%%%%%%%%%%%%%%%%%%%%%%%%%%%%%%%%%%%%%%%%%%%%%%%%%%%%%
\begin{thmbis}{th:RAFW_expectation_result}  Consider the set $\mathcal{M}=\conv(\mathcal{A})$, with $\mathcal{A}$ a finite set of extreme atoms, after $T$ iterations of Algorithm~\ref{algo:Randomized_away_general} (RAFW) we have the following linear convergence rate
\begin{eqnarray}
\mathbb{E}\big[ h(\xx_{T+1})\big]\leq \big(1- \eta^2 \rho_f\big)^{\max\{0,\floor{{(T-s)}/{2}}\}} h(\xx_{0})~,
\end{eqnarray}
with $\rho_f=\frac{\mu_f^A}{4C_f^A}$, $\eta=\frac{p}{|\mathcal{A}|}$ and $s=|\mathcal{S}_0|$.
\end{thmbis}

%\begin{thmbis}{th:RAFW_expectation_result} Suppose $f$ has bounded smoothness constant\footnote{The usual assumption of Lipschitz continuity of the gradient over compact domain result in a bounded smoothness constant.} $C_f^A$ as well as geometric strong convexity constant $\mu_f^A$.  Consider the set $\mathcal{M}=\text{conv}(\mathcal{A})$, with $\mathcal{A}$ a finite set of extrem atoms. 
%Let $s \defas \text{supp}(\xx_0)$.
%Then after $T$ iterations of RAFW (Algorithm \ref{algo:Randomized_away_general}), with a $p$ parameter of sub-sampling, we have
%\begin{eqnarray}\label{eq:RAFW_bound_result}
%\mathbb{E}\big[ h(\xx_{T+1})\big]\leq \big(1- \eta^2 \rho_f\big)^{\max\{0,\floor{{(T-s)}/{2}}\}} h(\xx_{0})~,
%\end{eqnarray}
%with $\rho_f=\frac{\mu_f^A}{4C_f^A}$, where $\eta=\frac{p}{|\mathcal{A}|}$ and $\mathbb{E}$ is a full expectation over all randomness until iteration $T$. 
%\end{thmbis}

%%%%%%%%%%%%%%%%%%%%%%%%%%%%%%%%%%%%%%%%%%%%%%%%%%%%%%%%%%%%%%%%%%%%%%%%%%%%%%%%%%%%%%%%%%%%%%%%%%%%%
%%%%%%%%%%%%%%%%%%%%%%%%%MAIN PROOF%%%%%%%%%%%%%%%%%%%%%%%%%%%%%%%%%%%%%%%%%%%%%%%%%%%%%%%%%%%%%%%%%%
%%%%%%%%%%%%%%%%%%%%%%%%%%%%%%%%%%%%%%%%%%%%%%%%%%%%%%%%%%%%%%%%%%%%%%%%%%%%%%%%%%%%%%%%%%%%%%%%%%%%%

\begin{proof}
The classical curvature constant used in proofs related to non-Away Frank-Wolfe is
\begin{eqnarray}\label{eq:old_curvature_constant}
C_f\coloneqq \underset{\substack{\xx,\ss\in\mathcal{M},\gamma\in [0,1] \\ \yy=\xx+\gamma (\ss-\vv)}}{\text{sup }} {\frac{2}{\gamma^2}\big( f(\yy)-f(\xx)-\langle\nabla f(\xx),\yy-\xx\rangle \big)}~.
\end{eqnarray}
It is tailored for algorithms in which the update is of the form $\xx_{t+1}=(1-\gamma)\xx_{t}+\gamma \vv_t$, but this is not the shape of all updates in away versions of FW. In \cite{lacoste2015global} they introduced a modification of the above curvature constant that we also use to analyze RAFW. It is defined in \citep[ equation (26)]{lacoste2015global} as
\begin{eqnarray}\label{eq:new_curvature_constant}
C_f^A\coloneqq \underset{\substack{\xx,\ss,\vv\in\mathcal{M},\gamma\in [0,1] \\ \yy=\xx+\gamma (\ss-\vv)}}{\text{sup }} {\frac{2}{\gamma^2}\big( f(\yy)-f(\xx)-\gamma\langle\nabla f(\xx),\ss-\vv\rangle \big)}~.
\end{eqnarray}
It differs from $C_f$ \eqref{eq:old_curvature_constant} because it allows to move outside of the domain $\mathcal{M}$. We thus require L-lipschitz continuous function on any compact set for that quantity to be upper-bounded. We refer to \S \textbf{curvature constants} on \citep[Appendix D]{lacoste2015global} for thorough details. The first part of the proof reuses the scheme of \citep[Theorem 8]{lacoste2015global}. 

%%%%%%%%%%%%%%%%%%%%%%%%%%%%%%%%%%%%%%%%%%%%%%%%%%%%%%%%%%%%%%%%%%%%%%%%%%%%%%%%%%%%%%%%%%%%%%%%%%%%%%%%
%%%%%%%%%%%%%%%%%FIRST PART: UPPER BOUNDING DIFF%%%%%%%%%%%%%%%%%%%%%%%%%%%%%%%%%%%%%%
%%%%%%%%%%%%%%%%%%%%%%%%%%%%%%%%%%%%%%%%%%%%%%%%%%%%%%%%%%%%%%%%%%%%%%%%%%%%%%%%%%%%%%%%%%%%%%%%%%%%%%%%
\noindent\textbf{First part.}
\textit{Upper bounding $h_t$}: Considering an iterate $\xx_{t}$ that is not optimal (e.g. $\xx_t \neq \xx^*$), from \citep[Eq. (28)]{lacoste2015global}, we have
\begin{eqnarray}\label{eq:upper_bound_h}
f(\xx_{t})-f(\xx^*)=h_t\leq \frac{\widetilde g_t^2}{2\mu_f^{A}}~,
\end{eqnarray}
where $\widetilde g_t$ is the \textit{pair-wise dual gap} defined by $\widetilde g_t=\langle \widetilde\ss_t-\vv_t,-\nabla f(\xx_t)\rangle$.  $\widetilde\ss_t$ and $\vv_t$ are respectively the FW atom and the away atom in the classical Away step algorithm (conditionally on $\xx_t$, the away atom of the randomized variant coincides with the away atom of the non-randomized variant). %We will later relate this to the partial pair-wise dual gap $g_t$, where the FW atom $\widetilde s_t$ has been replaced by the randomized FW atom $s_t$ coming from the sub-sampled oracle.
The result is still valid here as it only uses the definition of the affine invariant version of the strong convexity parameter and does not depend on the way $\xx_t$ are constructed (see \textit{upper bounding $h_t$} in \citep[Proof for AFW in Theorem 8]{lacoste2015global}).\\

Note that this implicitly assumes the away atom to be defined, e.g. the support of the iterate $\xx_t$ never to be zero. This is ensured by the algorithm simply because it always does convex updates.

%%%%%%%%%%%%%%%%%%%%%%%%%%%%%%%%%%%%%%%%%%%%%%%%%%%%%%%%%%%%%%%%%%%%%%%%%%%%%%%%%%%%%%%%%%%%%%%%%%%%%%%%
%%%%%%%%%%%%%%%%%SECOND PART: LOWER BOUNDING DIFF%%%%%%%%%%%%%%%%%%%%%%%%%%%%%%%%%%%%%%
%%%%%%%%%%%%%%%%%%%%%%%%%%%%%%%%%%%%%%%%%%%%%%%%%%%%%%%%%%%%%%%%%%%%%%%%%%%%%%%%%%%%%%%%%%%%%%%%%%%%%%%%
\noindent\textbf{Second part.}
\textit{Lower bounding progress $h_t-h_{t+1}$}. Consider $\xx_t$ a non-optimal iterate. At step $t$, the update in Algorithm \ref{algo:Randomized_away_general} writes $\xx_{t+1}(\gamma)=\xx_t+\gamma \dd_t$. $\gamma$ is optimized by line-search in the segment $[0,\gamma_{\text{max}}]$. Because in either cases $\dd_t$ is a difference between two elements of $\mathcal{M}$, from the definition of $C_f^A$ and because of the exact line search, we have
\begin{eqnarray*}
f(\xx_{t+1}) \leq \underset{\gamma\in [0,\gamma_{\text{max}}]}{\min} \big(f(\xx_t) + \gamma \langle \nabla f(\xx_t),\dd_t\rangle +\frac{\gamma^2}{2} C_f^{\mathcal{A}} \big)~,
\end{eqnarray*}
so that for any $\gamma\in[0;\gamma_{\text{max}}]$
\[
f(\xx_{t+1})-f(\xx_t) \leq \gamma \langle \nabla f(\xx_t),\dd_t\rangle +\frac{\gamma^2}{2} C_f^{\mathcal{A}}
\]
or again
\begin{eqnarray}\label{eq:curvature_inequality}
\gamma\frac{g_t}{2}-\frac{\gamma^2}{2}C^A_f \leq f(\xx_t)-f(\xx_{t+1}),
\end{eqnarray}
where the last inequality is a consequence of Lemma \ref{lem:progress_direction_sub_away}.
We write $\gamma_t^B \defas \frac{g_t}{2C_f^A}\geq 0$, the minimizer of the left hand side of \eqref{eq:curvature_inequality}.

\noindent\textbf{Case $\gamma_{\text{max}}\geq 1$ and $\gamma_t^B\leq \gamma_{\text{max}}$}. \eqref{eq:curvature_inequality} evaluated on $\gamma = \gamma_t^B$ gives
\begin{eqnarray}
\frac{g_t^2}{4C_f^A}-\frac{g_t^2}{8C_f^A}\leq f(\xx_t)-f(\xx_{t+1})\nonumber\\
\implies\big(\frac{g_t}{\widetilde{g}_t}\big)^2\frac{\widetilde{g}_t^2}{8C_f^A} \leq h_t-h_{t+1}.\label{eq:initial_inequality_ht_ht1}
\end{eqnarray}
Indeed, $\xx_t$ is assumed not to be optimal, so that $\widetilde{g}_t\neq 0$. Combining \eqref{eq:initial_inequality_ht_ht1} with \eqref{eq:upper_bound_h} gives
\begin{align}
h_{t+1}&\leq h_t - \big(\frac{g_t}{\widetilde{g}_t}\big)^2\frac{\widetilde{g}_t^2}{8C_f^A} \\
&\leq h_t - \big(\frac{g_t}{\widetilde{g}_t}\big)^2 \frac{\mu_f^A}{4C_f^A} h_t\\
&=\big( 1-\rho_f \big(\frac{g_t}{\widetilde{g}_t}\big)^2 \big) h_t~.
\end{align}

\noindent\textbf{Case $\gamma_{\text{max}}\geq 1$ and $\gamma_t^B > \gamma_{\text{max}}$}. $\gamma_t^B=\frac{g_t}{2 C_f^{\mathcal{A}}}$ implies $g_t\geq 2 C_f^A$. 
\eqref{eq:curvature_inequality} transforms into
\begin{eqnarray*}
\frac{g_t}{2}\big(\gamma-\frac{\gamma^2}{2}\big)\leq f(\xx_t)-f(\xx_{t+1})\nonumber\\
\frac{g_t}{\widetilde{g}_t}\frac{\widetilde{g}_t}{2}\big(\gamma-\frac{\gamma^2}{2}\big)\leq f(\xx_t)-f(\xx_{t+1})~.\nonumber\\
\end{eqnarray*}
Using $\widetilde{g}_t\geq h_t$ and evaluating at $\gamma=1$, leaves us with
\begin{eqnarray}
h_{t+1}\leq \big( 1-\frac{1}{4}  \frac{g_t}{\widetilde{g}_t}  \big) h_t.
\end{eqnarray}

Because $\mu_f^A\leq C_f^A$ \citep[Remark 7.]{lacoste2015global} and $\rho_f=\frac{\mu_f^{A}}{4C_f^A}$, the two previous cases resolve in the following inequality
\begin{eqnarray}\label{eq:recurrence_ht}
h_{t+1}\leq \big( 1-{\rho_f} \big(\frac{g_t}{\widetilde{g}_t}\big)^2 \big) h_t~.
\end{eqnarray}

\noindent\textbf{Case $\gamma_{\text{max}}< 1$ and $\gamma_t^*<\gamma_{\text{max}}$}. By definition
\begin{eqnarray}
\gamma_t^*&=&\underset{\gamma\in[0,\gamma_{\text{max}}]}{\argmin~}f(\xx_t+\gamma \dd_t)=F(\gamma)~.
\end{eqnarray}
$f$ is convex and its minimum on $[0;\gamma_{\text{max}}]$ is not reached at $\gamma_{\text{max}}$. It is then also a minimum on the interval $[0;+\infty]$, and in particular we have
\begin{eqnarray}
\gamma_t^*&=&\underset{\gamma\in[0,1]}{\argmin~}f(\xx_t+\gamma \dd_t)=F(\gamma)~.
\end{eqnarray}
\eqref{eq:curvature_inequality} can then be written with $\gamma\in[0,1]$ which leads to the previous case result \eqref{eq:recurrence_ht}.

\noindent\textbf{Case $\gamma_{\text{max}}< 1$ and $\gamma_t^*=\gamma_{\text{max}}$}. This corresponds to a particular drop step for which we only guarantee $h_{t+1}\leq h_t$ (exact line-search). We call this case a \textit{bad drop step} (indeed $\gamma_{\text{max}}>1$ and $\gamma_t^*=\gamma_{\text{max}}$ also corresponds to a drop step, but for which we can prove a bound of the form $h_{t+1}\leq h_t (1-\rho_f \big(\frac{g_t}{\widetilde{g}_t}\big)^2)$).\\

We use the binary indicator $z_t$ to distinguish between the step where \eqref{eq:recurrence_ht} is guaranteed or not. %When it is not guaranteed, it corresponds to the so-called \textit{bad drop step}. 
Denote by $z_t=0$ when doing a \textit{bad drop step} and $z_t=1$ otherwise. The second part can be summed-up in
\begin{eqnarray}\label{eq:fruit_second_part}
h_{t+1}\leq h_t (1-\rho_f \big(\frac{g_t}{\widetilde{g}_t}\big)^2)^{z_t}.
\end{eqnarray}
%And we immediately get by recurrence
%\begin{eqnarray}\label{eq:dual_improvement_without_expectation}
%h_{t+1}\leq h_0 \prod_{k=0}^{t}\big( 1-\rho_f   \big(\frac{g_k}{\tilde{g}_k}\big)^2  \big)^{z_k}.
%\end{eqnarray}

%%%%%%%%%%%%%%%%%%%%%%%%%%%%%%%%%%%%%%%%%%%%%%%%%%%%%%%%%%%%%%%%%%%%%%%%%%
%%%%%%%%%%%%%%%%%%%%%%%%%%%%%%%%%%%%%%%%%%%%%%%%%%%%%%%%%%%%%%%%%%%%%%%%%%
%%%%%%%%%%%%%%%%%LAST PART%%%%%%%%%%%%%%%%%%%%%%%%%%%%%%%%%%%%%%%%%%%%%%%%
%%%%%%%%%%%%%%%%%%%%%%%%%%%%%%%%%%%%%%%%%%%%%%%%%%%%%%%%%%%%%%%%%%%%%%%%%%
%%%%%%%%%%%%%%%%%%%%%%%%%%%%%%%%%%%%%%%%%%%%%%%%%%%%%%%%%%%%%%%%%%%%%%%%%%
\paragraph{Last part.}
%%%NUMBER DROP STEPS
Consider starting RAFW (Algorithm \ref{algo:Randomized_away_general}) for $T$ iterations at $\xx_0\in\text{conv}(\mathcal{V})$, with $s = |\mathcal{S}_0|\geq 0$. We will now prove there are at most $\left\lfloor\frac{T+s}{2}\right\rfloor$ drop steps. Let $D_T$ be the number of drop steps after iteration $T$ and $F_T$ the number of FW step adding a new atom until iteration $T$. By definition, a FW step is not a drop step so that $D_T+F_T\leq T$. Also $|S_T|=|S_0|+|F_T|-|D_T|$, hence $|\mathcal{S}_T|\leq |\mathcal{S}_0|-2|D_T|+T$ so that $|D_T|\leq \frac{T+s-|\mathcal{S}_T|}{2}$. Finally because $|\mathcal{S}_T|\geq 0$, we have  $|D_T|\leq \left\lfloor\frac{T+s}{2}\right\rfloor$.

%%%RECURRENCE PROPERTY TO PROVE
From the first two parts of the main proof, we have that
\begin{eqnarray}
h_{T}\leq h_0 \prod_{t=0}^{T-1}{\big(1-\rho_f\big(\frac{g_t}{\widetilde g_t}\big)^2\big)^{z_t}},
\end{eqnarray}
where $(g_t,z_t)_{t\in[0:T-1]}$ are defined along RAFW starting at $\xx_0$. For $i<j$, we write  $\mathbb{E}_{i:j}$ the expectation with respect to all sub-sampling between the $i^{th}$ iteration and the $j^{th}$ iteration included. When taking expectation only over sub-sampling $i$, we write it $\mathbb{E}_i$.

We will now prove by recurrence on $T\in\mathbb{N}^*$ that
\begin{eqnarray}\label{eq:recurrence_property}
\mathbb{E}_{0:T-1}(\prod_{t=0}^{T-1}{\big(1-\rho_f\big(\frac{g_t}{\widetilde g_t}\big)^2\big)^{z_t}})\leq (1-\rho_f \eta^2)^{\max\{0,T-\left\lfloor\frac{T+s}{2}\right\rfloor\}}=F(T,s)~~~~\forall s\in\mathbb{N}~~\forall \xx_0\in\mathbb{R}^d~~\text{with }~|\mathcal{S}_0|=s~,
\end{eqnarray}
where $\xx_0 = \sum_{\vv\in\mathcal{A}}{\alpha_{\vv}^{(0)}\vv}$ and $\mathcal{S}_0=\{\vv \in \mathcal{A} \text{ s.t. } \alpha_\vv^{(0)} > 0\}$.

The rate quantity $\max\{0,T-\left\lfloor\frac{T+s}{2}\right\rfloor\}$ represents the number of steps (between iteration $0$ and $T-1$) in which $z_t=1$, e.g. the steps in which there is a possibility of having geometrical decrease. Note that the geometrical decrease happens only when $g_t =\widetilde g_t$.

The key insight in the global bound is to recall (from section~\ref{s:rafw}) that if the support is a singleton, i.e. $|\mathcal{S}_t| =1$, RAFW does a FW step hence $z_t = 1$. We consequently distinguish whether or not the first iterate has an initial support of size $1$. We then use the recurrence property starting the algorithm at $\xx_1$ and running $T-1$ iterations.

%%%INITIALIZATION
\noindent\textbf{Initialization.} We will now prove the recurrence property \eqref{eq:recurrence_property} for $T=1$. If $s\geq 2$, $\max\{0,T-\left\lfloor\frac{T+s}{2}\right\rfloor\}=0$ and \eqref{eq:recurrence_property} is true because $(1-\rho_f\big(\frac{g_0}{\widetilde g_0}\big)^2\big)\leq 1$. If $s=1$, this implies that the first step needs to be a Frank-Wolfe step. We necessarily have $z_0=1$ and so
\begin{eqnarray}
\mathbb{E}_{0}(\big(1-\rho_f\big(\frac{g_0}{\widetilde g_0}\big)^2\big)^{z_0})&=&\mathbb{E}_{0}(\big(1-\rho_f\big(\frac{g_0}{\widetilde g_0}\big)^2\big)~\mid~z_0=1)\\
&\leq& 1-\rho_f\mathcal{P}(g_0=\widetilde g_0~\mid~z_0=1)\\
&\leq& 1-\rho_f \eta^2 \leq 1 \leq F(1,1)~,
\end{eqnarray}
with $\eta=\frac{p}{|\mathcal{A}|}$ where $F$ is defined in \eqref{eq:recurrence_property} and where the last inequality follows from \eqref{eq:proba} in Lemma \ref{lemma:proba_conditional_non_drop_step}.

%%%HEREDITY
\noindent\textbf{Recurrence.} Consider the property \eqref{eq:recurrence_property} when running $T-1$  iteration. By the tower property of conditional expectations

\begin{eqnarray}\label{eq:tower_property}
\mathbb{E}_{0:T-1}(\prod_{t=0}^{T-1}{\big(1-\rho_f\big(\frac{g_t}{\widetilde g_t}\big)^2\big)^{z_t}}) &=& \mathbb{E}_{0:T-1}\big[ 
\big(1-\rho_f\big(\frac{g_0}{\widetilde g_0}\big)^2\big)^{z_0}
\mathbb{E}_{1:T-1}(\prod_{t=1}^{T-1}{\big(1-\rho_f\big(\frac{g_t}{\widetilde g_t}\big)^2\big)^{z_t}}) \big].
\end{eqnarray}

We can apply the recurrence property with $T-1$ iterations and starting point $\xx_1$ on $\mathbb{E}_{1:T-1}(\prod_{t=1}^{T-1}{\big(1-\rho_f\big(\frac{g_t}{\widetilde g_t}\big)^2\big)^{z_t}})$ so that
\begin{eqnarray}
\mathbb{E}_{0:T-1}(\prod_{t=0}^{T-1}{\big(1-\rho_f\big(\frac{g_t}{\widetilde g_t}\big)^2\big)^{z_t}}) &\leq & \mathbb{E}_{0}\big[  \big(1-\rho_f\big(\frac{g_0}{\widetilde g_0}\big)^2\big)^{z_0}F(T-1,|\mathcal{S}_1|) 
\big]~,
\end{eqnarray}
where $|\mathcal{S}_1|$, the support of $\xx_1$, depends on $z_0$. Indeed $z_0=0$ implies a drop step and as such it decreases the support of the iterate. Thus we have to distinguish the case according to the size of the support of $\xx_0$.

% FIRST CASE DIVISION IN HEREDITY
\textbf{Case $|\mathcal{S}_0| = 1$.} With $\xx_0=0$, RAFW starts with a FW step and as such $z_0=1$ as well as $2\geq |\mathcal{S}_1|\geq 1$ so that
\begin{eqnarray}
\mathbb{E}_{0:T-1}(\prod_{t=0}^{T-1}{\big(1-\rho_f\big(\frac{g_t}{\widetilde g_t}\big)^2\big)^{z_t}}) &=& \mathbb{E}_{0}\big[  \big(1-\rho_f\big(\frac{g_0}{\widetilde g_0}\big)^2\big)
~|~z_0=1\big]F(T-1,|\mathcal{S}_1|)\\
&\leq& (1-\rho_f \eta^2)F(T-1,2)\leq F(T,1)~,
\end{eqnarray}
by applying \eqref{eq:proba} in Lemma \ref{lemma:proba_conditional_non_drop_step}. The last equality concludes the heredity in that case. 

% SECOND CASE DIVISION IN HEREDITY
\textbf{Case $|\mathcal{S}_0| \geq 2$.} Here it is possible for $z_0$ to equal $0$ or $1$. If $z_0=1$, then $|\mathcal{S}_1|\leq |\mathcal{S}_0|+1$, while if $z_0=0$, it implies a drop step, we have $|\mathcal{S}_1| = |\mathcal{S}_0|-1$. If we decompose the expectation according to the value of $z_0$ we obtain
\begin{eqnarray}
\mathbb{E}_{0:T-1}(\prod_{t=0}^{T-1}{\big(1-\rho_f\big(\frac{g_t}{\widetilde g_t}\big)^2\big)^{z_t}}) &\leq& 
\mathcal{P}(z_0=1)\mathbb{E}_{0}\big[  \big(1-\rho_f\big(\frac{g_0}{\widetilde g_0}\big)^2\big)~|~z_0=1\big]F(T-1,|\mathcal{S}_1|)\\
&&+\mathcal{P}(z_0=0) F(T-1,|\mathcal{S}_0|-1)\\
&\leq& 
\mathcal{P}(z_0=1)\big(1-\rho_f\eta^2\big)F(T-1,|\mathcal{S}_0|+1)+\mathcal{P}(z_0=0) F(T-1,|\mathcal{S}_0|-1)\\
&\leq& 
\mathcal{P}(z_0=1) \big(1-\rho_f\eta^2\big)F(T-1,s+1)+\mathcal{P}(z_0=0) F(T-1,s-1)~.
\end{eqnarray}
We used the fact that $F(T, |\mathcal{S}_1|) \leq F(T-1, |\mathcal{S}_0|+1)$. Since we do not have access to the values of $\mathcal{P}(z_0 = 0)$ and $\mathcal{P}(z_0 = 1)$, we bound it in the following manner
%\begin{eqnarray}
%\mathbb{E}_{0:T-1}(\prod_{t=0}^{T-1}{\big(1-\rho_f\big(\frac{g_t}{\widetilde g_t}\big)^2\big)^{z_t}}) &\leq& \min\big((1-\rho_f\eta^2)F(T-1,s+1), F(T-1,s-1)\big) = F(T,s)~,
%\end{eqnarray}
\begin{eqnarray}
\mathbb{E}_{0:T-1}(\prod_{t=0}^{T-1}{\big(1-\rho_f\big(\frac{g_t}{\widetilde g_t}\big)^2\big)^{z_t}}) &\leq& \max\big((1-\rho_f\eta^2)F(T-1,s+1), F(T-1,s-1)\big) \leq F(T,s)~,
\end{eqnarray}
where the last inequality is just about writing the definition of $F$. It concludes the heredity result.

\textbf{Conclusion:} Starting RAFW at $\xx_0$, after $T$ iterations, we have
\begin{eqnarray}
h_{T}\leq h_0 \prod_{t=0}^{T-1}{\big(1-\rho_f\big(\frac{g_t}{\widetilde g_t}\big)^2\big)^{z_t}}~.
\end{eqnarray}
Applying \eqref{eq:recurrence_property} we get
\begin{eqnarray}
\mathbb{E}_{0:T-1}(h_{T})&\leq& h_0 (1-\rho_f \eta^2)^{\max\{0,T-\left\lfloor\frac{T+s}{2}\right\rfloor\}}\nonumber\\
&\leq& h_0 (1-\rho_f\eta^2)^{\max\{0,\left\lfloor \frac{T-s}{2} \right\rfloor\}}~.
\end{eqnarray}

\end{proof}

%%%%%%%%%%%%%%%%%%%%%%%%%%%%%%%%%%%%%%%% GENERALIZED STRONGLY CONVEX%%%%%%%%%%%%%%%%%%%%%%%%%%
\paragraph{Generalized strongly convex.}

%\begin{thmbis}{th:Away_CV_Generalized_Strongly_Convex} Let $f$ be $L$-smooth and $\tilde{\mu}$-generally-strongly convex. Then, after $t$ iterations of Algorithm~\ref{algo:Randomized_away_general}, with a $p$ parameter of sub-sampling, we have
%\begin{eqnarray}
%\mathbb{E}\big[ h(\xx_{t+1})\big]\leq \big(1- \eta^2 \tilde{\rho}_f\big)^{\floor{\frac{t-1}{2}}} h(\xx_{0}),
%\end{eqnarray}
%with\footnote{The rate for away Frank-Wolfe in \citep[Theorem 8]{lacoste2015global}} $\tilde{\rho}_f=\frac{ \tilde{\mu}}{4 C_f^A}$ and $\eta=\frac{p}{|\mathcal{A}|}$.
%\end{thmbis}

\begin{thmbis}{th:Away_CV_Generalized_Strongly_Convex}  Suppose $f$ has bounded smoothness constant $C_f^A$ and is  $\tilde{\mu}$-generally-strongly convex.  Consider the set $\mathcal{M}=\text{conv}(\mathcal{A})$, with $\mathcal{A}$ a finite set of extreme atoms. Then after $T$ iterations of Algorithm \ref{algo:Randomized_away_general}, with $s=|\mathcal{S}_0|$ and a $p$ parameter of sub-sampling, we have
\begin{eqnarray}
\mathbb{E}\big[ h(\xx_{T+1})\big]\leq \big(1- \eta^2 \tilde{\rho}_f\big)^{\max\{0,\floor{\frac{T-s}{2}}\}} h(\xx_{0})~,
\end{eqnarray}
with $\tilde{\rho}_f=\frac{ \tilde{\mu}}{4 C_f^A}$ and $\eta=\frac{p}{|\mathcal{A}|}$.
\end{thmbis}

%\begin{thmbis}{th:Away_CV_Generalized_Strongly_Convex}  Suppose $f$ has bounded smoothness constant $C_f^A$ as well as $\tilde{\mu}$-generally-strongly convex.  Consider the set $\mathcal{M}=\text{conv}(\mathcal{A})$, with $\mathcal{A}$ a finite set of extreme atoms. Let's start RAFW (Algorithm \ref{algo:Randomized_away_general}) at $\xx_0\in\mathbb{R}^d$, with $\text{supp}(\xx_0)=s$. After $T$ iterations, with a $p$ parameter of sub-sampling, we have
%\begin{eqnarray}\label{eq:RAFW_convergence_rate_generally_strongly_convex}
%\mathbb{E}\big[ h(\xx_{T+1})\big]-h(\xx^*)\leq \big(1- \eta^2 \tilde{\rho}_f\big)^{\max\{0,\floor{\frac{T-s}{2}}\}} h(\xx_{0})~,
%\end{eqnarray}
%with $\tilde{\rho}_f=\frac{ \tilde{\mu}}{4 C_f^A}$ and $\eta=\frac{p}{|\mathcal{A}|}$. $\tilde{\mu}$ is meticulously defined in \citep[equation (39)]{lacoste2015global}.
%\end{thmbis}

\begin{proof}
The conclusion of proof of \citep[Th. 11]{lacoste2015global} is that we have similarly as equation \eqref{eq:upper_bound_h} by: 
\begin{eqnarray}
f(\xx_{t})-f(\xx^*)=h_t\leq \frac{g_t^2}{2\tilde{\mu}_f}~,
\end{eqnarray}
where $\tilde{\mu}_f>0$ is a similar measure of the affine invariant strong convexity constant but for generalized strongly convex function.

We can thus write the twin of equation \eqref{eq:fruit_second_part}
%\begin{eqnarray}\label{eq:dual_improvement_without_expectation_strongly_convex}
%h_{t+1}\leq h_0 \prod_{k=0}^{t}\big( 1-\tilde{\rho}_f   \big(\frac{g_t}{\widetilde{g}_t}\big)^2  \big)^{z_t},
%\end{eqnarray}
\begin{eqnarray}\label{eq:dual_improvement_without_expectation_strongly_convex}
h_{t+1}\leq h_t \big( 1-\tilde{\rho}_f   \big(\frac{g_t}{\widetilde{g}_t}\big)^2  \big)^{z_t}~,
\end{eqnarray}
with $\tilde{\rho}_f=\frac{\tilde{\mu}_f}{4 C_f^A}$. The rest of the proof follows is the same as that of Theorem~\ref{th:RAFW_expectation_result}.
\end{proof}

\clearpage

\section{Technical issues of previous work}\label{s:appendix_incomplete_frandi}

In this section we highlight some technical issues present in previous work.

\subsection{Randomized Frank-Wolfe in \citet{Frandi2016}}

\citet{Frandi2016} present a Randomized FW algorithm for the case of the $\ell_1$ ball in $\mathbb{R}^d$. Denote by $\mathcal{A}=\{\pm \boldsymbol{e}_i~~\forall i\in [d]\}$, where $\boldsymbol{e}_i$ is the canonical basis (i.e., the vector that is zero everywhere except on the $i$-th coordinate, where it equals one) the extremes atoms of the $\ell_1$ ball. Up to the iterative explicit implementation of the \textit{residuals}, \citep[Algorithm 2]{Frandi2016} with the sampling size $p\in[n]$ and our RFW (Algorithm \ref{algo:Randomized_FW_general}) are equivalent for the following choice of $\mathcal{A}_t$ in RFW
\begin{equation}
\mathcal{A}_t=\{\pm e_i~~\forall i\in \mathcal{I}_p\}~,~\text{ where $\mathcal{I}_p$ is random subset of $[d]$ of size $p$.}
\end{equation}

\paragraph{Convergence result.}
In this case, \citep[Proposition 2]{Frandi2016} gives the following convergence bound in expectation after $t$ iterations: 
\begin{eqnarray}\label{eq:expectation_rate_False}
\mathbb{E}(f(\xx_{t}))-f(\xx^*)\leq \frac{4 C_f}{t+2}~.
\end{eqnarray}
First, it is rather surprising that, unlike in our Theorem~\ref{th:expectation_rate}, the sub-sampling size $p$ does not appear in the convergence bound. A closer inspection at their Lemma $2$ reveals some errors in their proof. For the remainder of this section we will use the notation in \citep{Frandi2016}.

% The aim of \citep[Lemma 2]{Frandi2016} is to prove \citep[equation (23)]{Frandi2016}
% \begin{eqnarray*}
% h_{k+1}=\mathbb{E}_{\mathcal{S}^{(k)}}\big[ h(\alpha_{\lambda}^{(k+1)})\big]\leq h_k-\lambda g_k +\lambda^2 C_f~.
% \end{eqnarray*}
% This is akin to the descent lemma in proof of first order algorithms. It is the same building block that is founded in most of the proofs of FW related algorithms. Similarly we have \eqref{eq:diff_with_frandi} with a factor $\eta=\frac{p}{d}$, that represents the cost of sub-sampling.

\paragraph{The point of interest.} The proof of their Proposition 2 starts with the following inequality derived from the curvature constant:
\begin{equation}
    f(\alpha_{\lambda}^{(k+1)})\leq f(\alpha^{(k)}) +\lambda \big(u^{(k)}-\alpha^{(k)}\big)^T\nabla f(\alpha^{(k)})+\lambda^2 C_f~.
\end{equation}
Then it is claimed that the following equation, Eq. (24) in their paper, is a direct consequence ``after some algebraic manipulations''
\begin{equation}\label{eq:frandi_problematic_equation}
\mathbb{E}_{\mathcal{S}^{(k)}}\big[f(\alpha_{\lambda}^{(k+1)})\big]\leq f(\alpha^{(k)}) +\lambda \mathbb{E}_{\mathcal{S}^{(k)}}\big[\big(u^{(k)}-\alpha^{(k)}\big)^T\widetilde\nabla_{\mathcal{S}^{(k)}} f(\alpha^{(k)})\big]+\lambda^2 C_f~.
\end{equation}
which is not clear unless $u^{(k)}$ is independent of the sampling set, something that is not verified given that it is chosen \emph{precisely  from} the sampling set.

\paragraph{Technical details.}
$\lambda$ being positive, for Eq.~\eqref{eq:frandi_problematic_equation} to be true, we should necessarily have the following
\begin{eqnarray}\label{eq:false_equation}
\mathbb{E}_{\mathcal{S}^{(k)}}\big[\big(u^{(k)}-\alpha^{(k)}\big)^T\nabla f(\alpha^{(k)})\big] \leq  \mathbb{E}_{\mathcal{S}^{(k)}}\big[\big(u^{(k)}-\alpha^{(k)}\big)^T\widetilde\nabla_{\mathcal{S}^{(k)}} f(\alpha^{(k)})\big]~.
\end{eqnarray}
$\alpha^{(k)}$ as well as $\nabla f(\alpha^{(k)})$ are deterministic with respect to the $\mathcal{S}^{(k)}$ sampling set so the previous equation is equivalent to
\begin{equation}
\mathbb{E}_{\mathcal{S}^{(k)}}\big[\big(u^{(k)}\big)^T\nabla f(\alpha^{(k)})\big]-\big(\alpha^{(k)}\big)^T\nabla f(\alpha^{(k)}) \leq  \mathbb{E}_{\mathcal{S}^{(k)}}\big[\big(u^{(k)}\big)^T\widetilde\nabla_{\mathcal{S}^{(k)}} f(\alpha^{(k)})\big] -\big(\alpha^{(k)}\big)^T\mathbb{E}_{\mathcal{S}^{(k)}}\big[\widetilde\nabla_{\mathcal{S}^{(k)}} f(\alpha^{(k)})\big]
\end{equation}
Since the sub-sampling of $\mathcal{S}^{(k)}$ is uniform and by definition of $\widetilde\nabla_{\mathcal{S}^{(k)}} f(\alpha^{(k)})$ in \citep[equation (14)]{Frandi2016} we have $\mathbb{E}_{\mathcal{S}^{(k)}}\big[\widetilde\nabla_{\mathcal{S}^{(k)}} f(\alpha^{(k)})\big]=\nabla f(\alpha^{(k)})$. Then \eqref{eq:false_equation} is equivalent to
\begin{eqnarray}
\mathbb{E}_{\mathcal{S}^{(k)}}\big[\big(u^{(k)}\big)^T\nabla f(\alpha^{(k)})\big] &\leq&  \mathbb{E}_{\mathcal{S}^{(k)}}\big[\big(u^{(k)}\big)^T\widetilde\nabla_{\mathcal{S}^{(k)}} f(\alpha^{(k)})\big]~.
\end{eqnarray}
Also by definition in \citep[equation (22)]{Frandi2016}, $u^{(k)}$ the FW atom has its support on $\mathcal{S}^{(k)}$ as well as from \citep[equation (6)]{Frandi2016} we have that $\big(u^{(k)}\big)^T\nabla f(\alpha^{(k)})<0$ . So that $\big(u^{(k)}\big)^T\nabla f(\alpha^{(k)})=\big(u^{(k)}\big)^T\nabla_{\mathcal{S}^{(k)}} f(\alpha^{(k)})$ and finally \eqref{eq:false_equation} is equivalent to
\begin{eqnarray}
\frac{|\mathcal{S}^{(k)}|}{p}\mathbb{E}_{\mathcal{S}^{(k)}}\big[\big(u^{(k)}\big)^T\widetilde\nabla_{\mathcal{S}^{(k)}} f(\alpha^{(k)})\big] &\leq&  \mathbb{E}_{\mathcal{S}^{(k)}}\big[\big(u^{(k)}\big)^T\widetilde\nabla_{\mathcal{S}^{(k)}} f(\alpha^{(k)})\big]~,
\end{eqnarray}
this last inequality being false in general because $\frac{|\mathcal{S}^{(k)}|}{p}<1$ and $\mathbb{E}_{\mathcal{S}^{(k)}}\big[\big(u^{(k)}\big)^T\widetilde\nabla_{\mathcal{S}^{(k)}} f(\alpha^{(k)})\big]\leq 0$.

%Denoting by $r_k=\nabla f(\alpha^{(k)})$ and $\tilde{r}_k=\tilde{\nabla}_{S^{(k)}} f(\alpha^{(k)})$, in their notation, it requires the following to be true:
%\begin{eqnarray*}
%E_k \left((u^{(k)}-\alpha^{(k)})^T r_k\right)&\leq & E_k \left((u^{(k)}-\alpha^{(k)})^T \tilde{r}_k\right)\\
%\Leftrightarrow E_k \left((u^{(k)}-\alpha^{(k)})^T r_k\right)&\leq & E_k \left((u^{(k)})^T \tilde{r}_k \right) - (\alpha^{(k)})^T r_k\\
%\Leftrightarrow  E_k \left((u^{(k)})^T r_t\right)&\leq & E_k \left((u^{(k)})^T \tilde{r}_t\right)
%\end{eqnarray*}
%Since $u^{(k)}$ is an atom of the $\ell_1$ ball, we have $(u^{(k)})^T r_t= \frac{S^{(k)}}{p}(u^{(k)})^T \tilde{r}_t$ and finally we should have the following inequality:
%\begin{eqnarray}\label{eq:last_proof_Randi_False}
%\frac{S^{(k)}}{p} E_k \left( (u^{(k)})^T \tilde{r}_t \right) \leq  E_k \left((u^{(k)})^T \tilde{r}_t\right).
%\end{eqnarray}
%The definition of $u^{(k)}$ satisfies:
%\begin{eqnarray}
%u^{(k)}=\underset{u\in \mathcal{V}}{\argmin } ~~ u^T \tilde{r}_t
%\end{eqnarray}
%where $\mathcal{V}$ is the set of extreme points of the $\ell_1$ ball. Its symmetry ensures that  $(u^{(k)})^T \tilde{r}_t$ is negative. Equation (\ref{eq:last_proof_Randi_False}) is then false.

\clearpage

\end{document}